\pdfoutput=1
\documentclass[a4paper,11pt]{amsart}
\PassOptionsToPackage{final,hidelinks}{hyperref}
\usepackage{ifdraft}
\usepackage{hyperref}
\usepackage{nima}
\usepackage{xnotes}
\ifdraft{\usepackage{screen}}{}
\usepackage{tikz}
\usepackage{tikz-cd}
\usetikzlibrary{math}
\usetikzlibrary{calc}
\usetikzlibrary{decorations.markings}
\usetikzlibrary{decorations.pathreplacing}
\tikzset{
  vertex/.style={circle,minimum size=3pt,inner sep=0,fill=black},
  overtex/.style={circle,minimum size=3pt,inner sep=0,fill=white,draw},
  train/.style={decoration={
    markings,
    mark=
      between positions 0 and 1 step 2pt
      with
      {
        \draw (0,-1pt) -- (0,1pt);
      }
    },postaction={decorate}},
  markit/.style={decoration={
    markings,
    mark=
      between positions 0 and 1 step 15pt
      with
      {
        \draw (-1pt,-1pt) -- (1pt,1pt);
        \draw (1pt,-1pt) -- (-1pt,1pt);
      }
    },postaction={decorate}}
}

\allowdisplaybreaks

\usepackage{amsaddr}

\title{Strongly Shortcut Spaces}
\author{Nima Hoda}
\thanks{This work was partially supported by the ERC grant GroIsRan
  and an NSERC Postdoctoral Fellowship.}
\address{DMA, École normale supérieure \\ Université
  PSL, CNRS \\ 75005 Paris, France \\ \ \\ Cornell University}
\email{nima.hoda@mail.mcgill.ca}
\date{\today}


\newcommand{\eqnlabel}[1]{\label{eqn:#1}}
\newcommand{\eqnref}[1]{\ref{eqn:#1}}
\newcommand{\peqnref}[1]{(\eqnref{#1})}

\DeclareMathOperator{\CAT}{CAT}
\DeclareMathOperator{\SL}{SL}

\newcommand{\pc}[1]{P^{\circ}_{#1}}

\begin{document}

\begin{abstract}
  We define the strong shortcut property for rough geodesic metric
  spaces, generalizing the notion of strongly shortcut graphs.  We
  show that the strong shortcut property is a rough similarity
  invariant.  We give several new characterizations of the strong
  shortcut property, including an asymptotic cone characterization.
  We use this characterization to prove that asymptotically $\CAT(0)$
  spaces are strongly shortcut.  We prove that if a group acts
  metrically properly and coboundedly on a strongly shortcut rough
  geodesic metric space then it has a strongly shortcut Cayley graph
  and so is a strongly shortcut group.  Thus we show that $\CAT(0)$
  groups are strongly shortcut.

  To prove these results, we use several intermediate results which we
  believe may be of independent interest, including what we call the
  Circle Tightening Lemma and the Fine Milnor-Schwarz Lemma.  The
  Circle Tightening Lemma describes how one may obtain a
  quasi-isometric embedding of a circle by performing surgery on a
  rough Lipschitz map from a circle that sends antipodal pairs of
  points far enough apart.  The Fine Milnor-Schwarz Lemma is a
  refinement of the Milnor-Schwarz Lemma that gives finer control on
  the multiplicative constant of the quasi-isometry from a group to a
  space it acts on.
\end{abstract}

\maketitle

\tableofcontents

\section{Introduction}

The study of interactions between nonpositive curvature and infinite
group theory have a long history dating back to the work of Max Dehn
on fundamental groups of surfaces in the early 20th century.  These
ideas have been developed in a variety of directions since that time
and have become particularly relevant in recent decades with the
emergence of geometric group theory.  Various theories of
nonpositively curved groups have been developed: small cancellation
groups, $\CAT(0)$ groups, cubulated groups, systolic groups, quadric
groups, etc.  However, while the case of \emph{negatively} curved
groups has been satisfactorily unified by Gromov's seminal work on
hyperbolic groups \cite{Gromov:1987}, to this date there is no
satisfactory general notion of a nonpositively curved group.

Strongly shortcut graphs were introduced in earlier work of this
author \cite{Hoda:shortcut_graphs:2022} as graphs satisfying a weak
notion of nonpositive curvature.  They were shown to unify a broad
family of graphs of interest in geometric group theory and metric
graph theory including hyperbolic graphs, standard Cayley graphs of
finitely generated Coxeter groups and $1$-skeletons of finite
dimensional $\CAT(0)$ cube complexes, systolic complexes and quadric
complexes \cite{Hoda:shortcut_graphs:2022,
  Haettel_Hoda_Petyt:Coarse_injectivity:2020}.  They are finitely
presented and have polynomial isoperimetric functions and so have
decidable word problem \cite{Hoda:shortcut_graphs:2022}.  Strongly
shortcut groups are defined as those groups admitting a proper and
cocompact action on a strongly shortcut graph
\cite{Hoda:shortcut_graphs:2022}.  They include a wide family of
groups satisfying various nonpositive curvature conditions, including
hyperbolic groups, Coxeter groups, cocompactly cubulated groups,
systolic groups, quadric groups, finitely presented small cancellation
groups, Helly groups, hierarchically hyperbolic groups and even the
discrete Heisenberg groups \cite{Hoda:shortcut_graphs:2022,
  Haettel_Hoda_Petyt:Coarse_injectivity:2020,
  hoda_przytycki:heisenberg:2020}.

A graph $\Gamma$ is \newterm{strongly shortcut} if, for some $K > 1$,
there is a bound on the lengths of cycles $\alpha \colon S \to \Gamma$
for which
$d_X\bigl(\alpha(p), \alpha(\bar p)\bigr) \ge \frac{1}{K} \cdot
\frac{|S|}{2}$ for every antipodal pair of points $p, \bar p \in S$.
By a theorem of this author, a graph $\Gamma$ is \newterm{strongly
  shortcut} if and only if, for some $K > 1$, there is a bound on the
lengths of the $K$-bilipschitz embedded cycles of $\Gamma$
\cite{Hoda:shortcut_graphs:2022}.  A result of Papasoglu implies that
strongly shortcut groups have simply connected asymptotic cones
\cite[page 793]{Papasoglu:AsymptoticCones:1996}.  By another result of
Papasoglu, this implies that strongly shortcut groups have linear
isodiametric functions \cite[page 805]{Papasoglu:AsymptoticCones:1996}
and, by a result of Riley, this implies that strongly shortcut groups
have linear filling length functions
\cite[Theorem~C]{Riley:Higher_Connectedness:2003}.

In this paper, we introduce a generalization of this notion to rough
geodesic metric spaces.  A metric space $X$ is \defterm{$R$-rough
  geodesic} if, for every $x_1, x_2 \in X$, there exists a function
$f \colon [0,\ell] \to X$ such that $f(0) = x_1$, $f(\ell) = x_2$,
$\ell = d(x_1,x_2)$ and
\[ |s-t| -R \le d\bigl(f(s),f(t)\bigr) \le |s-t| + R\] for any $s,t$
in the interval $[0,\ell]$.  This is the same as $X$ being
$(1,R)$-quasi-geodesic.  The special case $R = 0$ is that of geodesic
metric spaces.  An $R$-rough geodesic metric space $X$ is
\newterm{strongly shortcut} if, for some $K > 1$, there is a bound on
the lengths of Riemannian circles $S$ for which there exists an
$R$-rough $1$-Lipschitz map $\alpha \colon S \to X$ that satisfies
$d_X\bigl(\alpha(p), \alpha(\bar p)\bigr) \ge \frac{1}{K} \cdot
\frac{|S|}{2}$ for every antipodal pair of points $p, \bar p \in S$.
Such a map $\alpha$ is called a \defterm{$\frac{1}{K}$-almost
  isometric $R$-circle}.  We give several characterizations of the
strong shortcut property, we show that a group acting metrically
properly and coboundedly on a strongly shortcut rough geodesic metric
space is a strongly shortcut group and we prove a few other results
that may be of independent interest in metric geometry and geometric
group theory.  The results of this paper are applied in several
upcoming papers of the present author and his coauthors
\cite{Hoda_Krishna:Relatively_hyperbolic:2020,
  Haettel_Hoda_Petyt:Coarse_injectivity:2020,
  hoda_przytycki:heisenberg:2020}.

Below is a summary of our main results.

\Mainthmref{sshortcut_equiv} below gives several conditions that are
equivalent to the strong shortcut property for rough geodesic metric
spaces.  Of particular note is condition
\pitmref{mainthm_no_ascone_circle} which expresses the strong shortcut
property in terms of asymptotic cones.  Conditions
\pitmref{mainthm_no_qie_r} and \pitmref{mainthm_no_qie} generalize
Proposition 3.5 of \cite{Hoda:shortcut_graphs:2022}, which expresses
the strong shortcut property for graphs in terms of bilipschitz
cycles.  These two generalizations have their own advantages: for
geodesic metric spaces (i.e. when $R = 0$) condition
\pitmref{mainthm_no_qie_r} expresses the strong shortcut property
purely in terms of bilipschitz maps from circles whereas condition
\pitmref{mainthm_no_qie} avoids the dependence on $R$.  Condition
\pitmref{mainthm_ngons_bounded} expresses the strong shortcut property
in terms of nonapproximability of certain finite metric spaces at
large scale.  Conditions \pitmref{mainthm_ngons_bounded} and
\pitmref{mainthm_no_ascone_circle} also make sense for general metric
spaces and we prove that they are equivalent for general metric
spaces.  (See \Thmref{general_metric_neg_equiv}.)  Thus one may
consider condition \pitmref{mainthm_ngons_bounded} of
\Mainthmref{sshortcut_equiv} as a definition of the strong shortcut
property for general metric spaces.

\begin{mainthm}[\Thmref{not_sshortcut_equiv}]
  \mainthmlabel{sshortcut_equiv} Let $X$ be an $R$-rough geodesic
  metric space.  The following conditions are equivalent.
  \begin{enumerate}
  \item \itmlabel{mainthm_is_strong_shortcut} $X$ is strongly shortcut.
  \item \itmlabel{mainthm_no_qie_r} There exists an $L > 1$ such that
    there is a bound on the lengths of the $(L,4R)$-quasi-isometric
    embeddings of Riemannian circles in $X$.
  \item \itmlabel{mainthm_no_qie} There exists an $L > 1$ such that
    for every $C \ge 0$ there is a bound on the lengths of the
    $(L,C)$-quasi-isometric embeddings of Riemannian circles in $X$.
  \item \itmlabel{mainthm_ngons_bounded} For some $L > 1$ and some
    $n \in \N$, there is a bound on the $\lambda > 0$ for which there
    exists an $L$-bilipschitz embedding of $\lambda S_n^0$ in $X$,
    where $S_n^0$ is the vertex set of the cycle graph $S_n$ of length
    $n$ and $\lambda S_n^0$ is $S_n^0$ with the metric scaled by
    $\lambda$.
  \item \itmlabel{mainthm_no_ascone_circle} No asymptotic cone of $X$
    contains an isometric copy of the Riemannian circle of unit
    length.
  \end{enumerate}
\end{mainthm}

The main difficulty in proving \Mainthmref{sshortcut_equiv} is in the
implication
$\pitmref{mainthm_no_qie_r} \Longrightarrow
\pitmref{mainthm_is_strong_shortcut}$.  This is because a
$\frac{1}{K}$-almost isomertric $R$-circle in $X$ does not need to be
an $(L,4R)$-quasi-isometric embedding for any $L > 1$: while the
almost isometric condition only concerns pairs of antipodal points,
the quasi-isometry condition concerns all pairs of points.  The idea
of the proof is that given a $\frac{1}{K}$-almost isometric $R$-circle
$\alpha$ with $K > 1$ sufficiently close to $1$, we can perform
surgery on $\alpha$ in order to obtain an $(L,4R)$-quasi-isometric
embedding where $L$ depends on $K$ in such a way that if $K \to 1$
then $L \to 1$ also.  The contrapositive
$\neg \pitmref{mainthm_is_strong_shortcut} \Longrightarrow \neg
\pitmref{mainthm_no_qie_r}$ then readily follows since any family of
arbitrarily long $R$-circles with almost isometric constant $K$
approaching $1$ could then be surgered to produce a family of
quasi-isometric embeddings with the multiplicative constant $L$
tending to $1$.  The circle surgery result, which we call the Circle
Tightening Lemma, is stated in slightly simplified form in
\Mainthmref{circle_tightening} below and expressed more formally in
\Lemref{circle_tightening}.

\begin{mainthm}[\Thmref{graph_neg_equiv}, \Thmref{not_sshortcut_equiv}]
  Let $\Gamma$ be a graph.  Then $\Gamma$ is strongly shortcut as a
  graph if and only if $\Gamma$ is strongly shortcut as a geodesic
  metric space.
\end{mainthm}

\Mainthmref{ss_group} below gives several conditions that are
equivalent to the strong shortcut property for groups.  Condition
\pitmref{has_ss_cayley} reduces the property to the existence of a
strongly shortcut Cayley graph.  The proof is a direct application of
the Fine Milnor-Schwarz Lemma (\Mainthmref{fine_milnor_schwarz} below)
and stability of the strong shortcut property under scaling and
quasi-isometric perturbation of the metric
(\Mainthmref{mainthm_rough_approx_inv} below).

\begin{mainthm}[\Corref{ss_group}]
  \mainthmlabel{ss_group} Let $G$ be a group.  The following
  conditions are equivalent
  \begin{enumerate}
  \item $G$ is strongly shortcut.
  \item $G$ acts metrically properly and coboundedly on a strongly
    shortcut rough geodesic metric space.
  \item \itmlabel{has_ss_cayley} $G$ has a finite generating set $S$
    for which the Cayley graph of $(G,S)$ is strongly shortcut.
  \end{enumerate}
\end{mainthm}

Asymptotically $\CAT(0)$ spaces and groups were first introduced and
studied by Kar \cite{Kar:2011}.  A metric space $X$ is
\newterm{asymptotically $\CAT(0)$} if every asymptotic cone of $X$ is
$\CAT(0)$.  A group is \newterm{asymptotically $\CAT(0)$} if it acts
properly and cocompactly on an asymptotically $\CAT(0)$ proper
geodesic metric space.  Examples of asymptotically $\CAT(0)$ spaces
include $\CAT(0)$ spaces, Gromov-hyperbolic spaces and
$\widetilde{\SL(2,\R)}$ with the Sasaki metric \cite{Kar:2011}.

\begin{mainthm}[\Thmref{as_cat0_space}, \Thmref{as_cat0_group}]
  Asymptotically $\CAT(0)$ rough geodesic metric spaces are strongly
  shortcut.  Consequently, (asymptotically) $\CAT(0)$ groups are
  strongly shortcut.
\end{mainthm}

\Mainthmref{ascone_stability} below shows that the strong shortcut
condition is preserved under taking asymptotic cones.  This was
suggested as a desirable property for a general notion of nonpositive
curvature in Gromov
\cite[Section~6.E]{Gromov:Asymptotic_Invariants:1993}.

\begin{mainthm}[\Corref{as_cones_ss}]
  \mainthmlabel{ascone_stability} Let $X$ be an $R$-rough geodesic
  metric space.  If $X$ is strongly shortcut then every asymptotic
  cone of $X$ is strongly shortcut.
\end{mainthm}

\Mainthmref{mainthm_rough_approx_inv} below has several consequences.
In addition to showing that the strong shortcut property descends to
isometric subspaces and is a rough similarity invariant, it implies
that for a given strongly shortcut space, a sufficiently small
bilipschitz distortion of the metric preserves the strong shortcut
property.  This is another property which is discussed in Gromov
\cite[Section~6.E]{Gromov:Asymptotic_Invariants:1993}.

\begin{mainthm}[\Corref{rough_approx_inv}]
  \mainthmlabel{mainthm_rough_approx_inv} Let $X$ be a strongly
  shortcut rough geodesic metric space.  Then there exists an
  $L_X > 1$ such that whenever $Y$ is a rough geodesic metric space
  and $C > 0$ and $f \colon Y \to X$ is an $(L_X,C)$-quasi-isometric
  embedding up to scaling, then $Y$ is also strongly shortcut.  In
  particular, the strong shortcut property is a rough similarity
  invariant of rough geodesic metric spaces.
\end{mainthm}

In fact, \Mainthmref{mainthm_rough_approx_inv} holds for general
metric spaces with condition \pitmref{mainthm_ngons_bounded} of
\Mainthmref{sshortcut_equiv} in place of the strong shortcut property.
(See \Propref{not_approx_ngon_invariant}.)  It should be noted that
the strong shortcut property is not a quasi-isometry invariant so one
cannot hope to remove the dependence on $X$ of the quasi-isometry
constant $L_X$.  See \Ssecref{not_qi_inv} for an example.

The following result, which we call the Circle Tightening Lemma,
states that a map from a circle that satisfies a rough lipchitz upper
bound and that, on antipodes, satisfies a bilipschitz lower bound can
be upgraded through surgery to a rough bilipschitz map.  See
\Figref{circle_tightening_intro}.  The Circle Tightening Lemma is
essential in the proof of \Mainthmref{sshortcut_equiv}.  We believe it
may be of independent interest.  Here we express a slightly simplified
version of the Circle Tightening Lemma.  For the formal statement,
please see \Lemref{circle_tightening}.

\begin{figure}[ht]
  \newcommand{\circrad}{2cm}

  \newcommand{\arc}[2]{
    \draw (#1:\circrad) arc[start angle=#1, end angle=#2, radius=\circrad]}
  \newcommand{\spike}[4]{
    \draw[red] (#2:\circrad) arc[start angle=#2, end angle=#3, radius=\circrad];
    \draw (#2:\circrad)
    to[out=#2/2+#3/2+90,in=#2/2+#3/2] (#2/2+#3/2:#1*\circrad);
    \draw (#3:\circrad)
    to[out=#2/2+#3/2-90,in=#2/2+#3/2] (#2/2+#3/2:#1*\circrad);
    \node at (#2/2+#3/2:1em+#1*\circrad) {#4}}

  \centering
  \begin{tikzpicture}
    \spike{1.3}{-20}{-5}{$Q_4$};
    \spike{1.12}{30}{35}{$Q_3$};
    \spike{1.1}{75}{80}{$Q_6$};
    \spike{1.2}{173}{183}{$Q_5$};
    \spike{1.05}{230}{232}{$Q_7$};
    \spike{1.1}{267}{273}{$Q_2$};
    \spike{1.15}{302}{309}{$Q_1$};
    
    \arc{-5}{30};
    \arc{35}{75};
    \arc{80}{173};
    \arc{183}{230};
    \arc{232}{267};
    \arc{273}{302};
    \arc{309}{340};
  \end{tikzpicture}
  
  \caption{The closed outer path in black is a $1$-Lipchitz embedding
    $\alpha$ of a Riemannian circle $S$.  This embedding $\alpha$ has
    poor bilipchitz constant but only because it badly distorts
    distances between relatively nearby pairs of points of $S$ (pairs
    contained in the subpaths $\alpha|_{Q_i}$).  If we consider only
    antipodal pairs of points of $S$ then the distortion of their
    distances under $\alpha$ is much less than in these worst cases.
    In other words $\alpha$ has low distortion when viewed
    \emph{globally}.  The Circle Tightening Lemma tells us that if the
    global distortion of $\alpha$ is low enough then we can perform
    surgery on $\alpha$, replacing distorted subpaths of arbitrarily
    low total relative length with efficient alternatives (the red
    paths in the figure) in order to obtain an arbitrarily good
    bilipschitz constant.}
  \figlabel{circle_tightening_intro}
\end{figure}

\begin{mainthm}[Circle Tightening Lemma, \Lemref{circle_tightening}]
  \mainthmlabel{circle_tightening}

  Let $X$ be an $R$-rough geodesic metric space with $R \ge 0$, let
  $L > 1$ and let $\epsilon > 0$.  There exists a $K > 1$ such that if
  $\alpha \colon S \to X$ is a sufficiently long $R$-rough
  $1$-Lipschitz map from a Riemannian circle $S$ satisfying
  \[ d_X\bigl(\alpha(p), \alpha(\bar p)\bigr) \ge \frac{1}{K} d_S(p,
    \bar p) \] for every antipodal pair $p, \bar p \in S$ then there
  exists a countable collection $\{Q_i\}_i$ of pairwise disjoint
  closed segments in $S$ of total length $\sum_i|Q_i| < \epsilon |S|$
  such that shortening the $Q_i$ and replacing the
  $\alpha|_{Q_i} \colon Q_i \to X$ we can obtain from $\alpha$ an
  $(L,4R)$-quasi-isometric embedding of a circle.
\end{mainthm}

Note that in the statement of the Circle Tightening Lemma, the rough
geodesicity constant $R$ may be equal to $0$ in which case the result
is about $1$-Lipschitz maps and $L$-bilipschitz maps in a geodesic
metric space.

We call the following refinement of the Milnor-Schwarz Lemma the Fine
Milnor-Schwarz Lemma.  It is used in the proof of
\Mainthmref{ss_group}.  It essentially says that if a group $G$ acts
metrically properly and coboundedly on a rough geodesic space $X$
then, up to scaling, the group $G$ has word metrics that are
quasi-isometric to $X$ with multiplicative constant arbitrarily close
to $1$.  We believe it may be of independent interest.

\begin{mainthm}[Fine Milnor-Schwarz Lemma, \Lemref{fine_ms}, \Rmkref{metric_proper_finite}]
  \mainthmlabel{fine_milnor_schwarz}
  
  Let $(X,d)$ be a rough geodesic metric space.  Let $G$ be a group
  acting metrically properly and coboundedly on $X$ by isometries.
  Fix $x_0 \in X$.  For $t > 0$ let $S_t$ be the finite set defined by
  \[ S_t = \bigl\{ g \in G \sth d(x_0,gx_0) \le t \bigr\} \] and
  consider the word metric $d_{S_t}$ defined by $S_t$.  (For those $t$
  where $S_t$ does not generate $G$, we allow $d_{S_t}$ to take the
  value $\infty$).  Let $K_t$ be the infimum of all $K > 1$ for which
  \begin{align*}
    (G,td_{S_t}) &\to X \\
    g &\mapsto g \cdot x_0
  \end{align*}
  is a $(K, C_K)$-quasi-isometry for some $C_K \ge 0 $.  Then
  $K_t \to 1$ as $t \to \infty$.
\end{mainthm}

\subsection{Structure of the paper}

In \Secref{definitions} we introduce basic notions that will be used
throughout the paper.  In \Secref{characterizations} prove various
characterizations of the strong shortcut property and prove that it is
a rough similarity invariant.  In \Secref{circle_tightening} we state
and prove the Circle Tightening Lemma.  In \Secref{fine_ms} we state
and prove the Fine Milnor-Schwarz Lemma.  In \Secref{as_cat0}, we
apply the results of the previous sections to prove that
asymptotically $\CAT(0)$ groups are strongly shortcut.

\subsection{Acknowledgements}

The author would like to thank Pierre Pansu for some valuable
discussions about asymptotic cones of strongly shortcut graphs.

\section{Basic notions and definitions}
\seclabel{definitions}

Let $X$ and $Y$ be metric spaces, let $S$ be a set and let $R$ be a
nonnegative real.  A function $f \colon S \to Y$ is $R$-roughly onto
if every $y \in Y$ is at distance at most $R$ from some point in
$f(X) \subset Y$.  An \defterm{$R$-rough isometric embedding} from $X$
to $Y$ is a function $f \colon X \to Y$ such that
\[ d(x_1,x_2) - R \le d\bigl(f(x_1),f(x_2)\bigr) \le d(x_1,x_2) + R \]
for all $x_1,x_2 \in X$.  An $R$-rough isometric embedding is the same
as a $(1,R)$-quasi-isometric embedding.  An $R$-rough isometric
embedding $f \colon X \to Y$ is an \defterm{$R$-rough isometry} if it
is roughly onto.  An $R$-rough isometry is the same as a
$(1,R)$-quasi-isometry.

An \defterm{$R$-rough geodesic} in $X$ from $x_1$ to $x_2$ is an
$R$-rough isometric embedding $f$ from the interval
$\bigl[0,\ell\bigr] \subset \R$ to $X$ with $\ell = d(x_1,x_2)$ such
that $f(0) = x_1$ and $f(\ell) = x_2$.  An $R$-rough geodesic is the
same as a $(1,R)$-quasi-geodesic.  A metric space $(X,d)$ is
\defterm{$R$-rough geodesic} if every pair of points in $X$ is joined
by an $R$-rough geodesic.\footnote{We include the condition
  $\ell = d(x_1,x_2)$ in the definition of an $R$-rough geodesic $f$
  only for convenience.  If we do not assume it, then we can recover
  it up to slightly increasing the rough geodesicity constant to
  $R' = (1+\sqrt{2})R$.  Indeed, the remaining conditions on $f$ imply
  that $\bigl|\ell - d(x_1,x_2)\bigr| \le R$ and it can be shown that
  either $d(x_1, x_2) \le R'$ (in which case any function
  $\bigl[0,d(x_1,x_2)\bigr] \to \{x_1,x_2\}$ is an $R'$-rough
  isometric embedding) or the composition of $f$ with the orientation
  preserving linear bijection $\bigl[0,d(x_1,x_2)\bigr] \to [0,\ell]$
  results in an $R'$-rough isometric embedding.}  Note that rough
geodesicity implies weak geodesicity, as used by Kasparov and
Skandalis and others
\cite{Kasparov_Skandalis:Groupes_boliques:1994,Kasparov_Skandalis:Groups_acting_bolic:2003,Lafforgue:K_theorie_bivariante:2002,Mineyev_Yu:Baum-Connes:2002}.
A natural question is whether or not every rough geodesic space can be
thickened, in the sense of Gromov
\cite[Section~1.B]{Gromov:Asymptotic_Invariants:1993} to a geodesic
metric space.

Let $X$ and $Y$ be metric spaces, let $R \ge 0$ and let $K \ge 1$.  An
\defterm{$R$-rough $K$-Lipschitz map} from $Y$ to $X$ is a function
$\alpha \colon Y \to X$ such that
\[ d\bigl(\alpha(p), \alpha(q)\bigr) \le K d(p, q) + R \] for all
$p,q \in Y$.  An \defterm{$R$-path} in $X$ is an $R$-rough
$1$-Lipschitz map $\alpha \colon P \to X$ from an interval
$P \subset \R$.  An \defterm{$R$-circle} in $X$ is an $R$-rough
$1$-Lipschitz map $\alpha \colon S \to X$ from a Riemannian circle
$S$.  We use the notation $|F|$ to denote the length of $F$, where $F$
is an interval, a Riemannian circle or a finite union of closed
segments in an interval or in a Riemannian circle.

\begin{rmk}
  \rmklabel{rpath_concatenation} The concatenation of two $R$-paths
  need not be an $R$-path.  However, if $\alpha_1 \colon P_1 \to X$
  and $\alpha_2 \colon P_2 \to X$ are a pair of concatenatable
  $R$-paths and $\gamma \colon [0,R] \to X$ is the constant path at
  the point of concatenation then the concatenation
  $\alpha_1\gamma\alpha_2$ is an $R$-path.
\end{rmk}

An $R$-circle $\alpha \colon S \to X$ is \defterm{$\frac{1}{K}$-almost
  isometric}, for some $K > 1$, if
\[ d\bigl(\alpha(p), \alpha(\bar p)\bigr) \ge \frac{1}{K} \cdot
  \frac{|S|}{2} \] for every antipodal pair of points
$p, \bar p \in S$.

\begin{defn}
  \defnlabel{sshortcut} An $R$-rough geodesic metric space $X$ is
  \defterm{strongly shortcut} if, for some $K > 1$, there is a bound
  on the lengths of the $\frac{1}{K}$-almost isometric $R$-cycles of
  $X$.
\end{defn}

\begin{rmk}
  By \Thmref{not_sshortcut_equiv}, the apparent dependence on $R$ in
  \Defnref{sshortcut} is not essential.  That is, if $X$ is an
  $R$-rough geodesic metric space and $R' > R$ then $X$ is strongly
  shortcut if and only if it is strongly shortcut when viewed as an
  $R'$-rough geodesic metric space.
\end{rmk}

We view graphs as geodesic metric spaces with each edge isometric to a
unit interval.  For a graph $\Gamma$, we use the notation $\Gamma^0$
to denote the vertex set of $\Gamma$ with its subspace metric.  The
\defterm{cycle graph} $S_n$ of length $n$ is the graph isometric to a
Riemannian circle of length $n$.  A \defterm{cycle} in a graph
$\Gamma$ is a combinatorial map $S_n \to \Gamma$ from some cycle graph
$S_n$ to $\Gamma$.  A \defterm{path graph} is a graph isometric to a
real interval.  A \defterm{combinatorial path} in a graph $\Gamma$ is
a combinatorial map $P \to \Gamma$ from a path graph $P$.

Note that if $\alpha \colon S_n \to \Gamma$ is a cycle in a graph
$\Gamma$ then $\alpha$ is a $1$-Lipschitz map from a Riemannian circle
to a geodesic metric space or, in the language we have established
above, an $R$-circle in an $R$-rough geodesic metric space where
$R = 0$.

\begin{defn}
  A graph $\Gamma$ is \defterm{strongly shortcut as a graph} if, for
  some $K > 1$, there is a bound on the lengths of the
  $\frac{1}{K}$-almost isometric cycles of $\Gamma$.
\end{defn}

\begin{rmk}
  \rmklabel{ss_asgraph_asmetric} By \Thmref{graph_neg_equiv},
  \Thmref{not_sshortcut_equiv} and \Corref{rough_approx_inv}, the
  following conditions are equivalent for a graph $\Gamma$.
  \begin{enumerate}
  \item $\Gamma$ is strongly shortcut as a graph.
  \item $\Gamma$ is strongly shortcut as a geodesic metric space.
  \item $\Gamma^0$ is strongly shortcut as a rough geodesic metric
    space.
  \end{enumerate}
\end{rmk}

If $X$ is a metric space and $\lambda > 0$ then we write $\lambda X$
to denote the metric space obtained from $X$ by scaling the metric by
$\lambda$.

\section{Characterizing the strong shortcut property}
\seclabel{characterizations}

In this section we will give various characterizations of the strong
shortcut property.

\begin{lem}
  \lemlabel{global_to_local} Let $\alpha \colon S \to X$ be a
  $\frac{1}{K}$-almost isometric $R$-circle in a metric space $X$.
  Then
  \[ d\bigl(\alpha(p),\alpha(q)\bigr) \ge d(p,q) - \frac{K-1}{K}
    \cdot \frac{|S|}{2} - 2R \] for all $p,q \in S$.
\end{lem}
\begin{proof}
  Let $p',q' \in S \setminus \{p,q\}$ be antipodal and suppose that a
  geodesic segment of $S$ visits $p'$, $p$, $q$, and $q'$, in that
  order.  Then
  \begin{align*}
    \frac{1}{K} \cdot \frac{|S|}{2}
    &\le d\bigl(\alpha(p'),\alpha(q')\bigr) \\
    &\le d\bigl(\alpha(p'),\alpha(p)\bigr)
      + d\bigl(\alpha(p),\alpha(q)\bigr)
      + d\bigl(\alpha(q),\alpha(q')\bigr) \\
    &\le d(p',p) + R + d\bigl(\alpha(p),\alpha(q)\bigr) + d(q,q') + R \\
    &= d(p',p) + d(q,q') + d\bigl(\alpha(p),\alpha(q)\bigr)  + 2R \\
    &= \frac{|S|}{2} - d(p,q) + d\bigl(\alpha(p),\alpha(q)\bigr)  + 2R
  \end{align*}
  from which we can obtain the desired inequality.
\end{proof}

The following definition is very useful because it applies to general
metric spaces.

\begin{defn}
  A metric space $X$ \defterm{approximates $n$-gons} if, for every
  $K > 1$, and every $n \in \N$ there exist $K$-bilipschitz embeddings
  of $\lambda S_n^0$ in $X$ for arbitrarily large $\lambda > 0$.
\end{defn}

We will see that in the case of a graph or a rough geodesic metric
space \emph{non}approximation of $n$-gons is equivalent to the strong
shortcut property.  Thus it would make sense to define the strong
shortcut property for general metric spaces as nonapproximability of
$n$-gons.

\begin{defn}
  Let $X$ and $Y$ be metric spaces.  A function $f \colon Y \to X$ is
  a $(K,C)$-quasi-isometry \defterm{up to scaling} if there exists a
  $\lambda > 0$ such that $f$ is a $(K,C)$-quasi-isometry when viewed
  as a function from $\lambda Y$ to $X$.  A function
  $f \colon Y \to X$ is a \defterm{rough similarity} if, for some
  $C > 0$, the function $f$ is a $(1,C)$-quasi-isometry up to scaling.
  A property $\mathscr{P}$ of metric spaces is a \defterm{rough
    similarity invariant} if whenever $X$ satisfies $\mathscr{P}$ and
  $f \colon Y \to X$ is a rough similarity then $Y$ also satisfies
  $\mathscr{P}$.  A property $\mathscr{P}$ of metric spaces is a
  \defterm{rough approximability invariant} if, for any metric space
  $X$ satisfying $\mathscr{P}$, there exists an $L_X > 1$ such that
  whenever $C > 0$ and $f \colon Y \to X$ is an
  $(L_X,C)$-quasi-isometric embedding up to scaling, then $Y$ also
  satisfies $\mathscr{P}$.
\end{defn}

\begin{prop}
  \proplabel{not_approx_ngon_invariant} Nonapproximability of $n$-gons
  is a rough approximability invariant of metric spaces.  In
  particular, nonapproximability of $n$-gons is a rough similarity
  invariant of metric spaces.
\end{prop}
\begin{proof}
  Let $X$ be a metric space that does not approximate $n$-gons.  Then
  there is a $K > 1$ and an $n \in \N$ and a $\Lambda > 0$ such that
  any $K$-bilipschitz embedding of $\lambda S_n^0$ in $X$ satisfies
  $\lambda < \Lambda$.
  
  Let $Y$ be a metric space, let $t > 0$, let $L \in (1,K)$ and let
  $f \colon tY \to X$ be an $(L,C)$-quasi-isometric embedding.  Let
  $K' \in \bigl(1, \frac{K}{L}\bigr)$ and let
  $\lambda' > \frac{CLK'}{t}$.  We will show that there is a bound on
  the $\lambda'$ for which there exists a $K'$-bilipschitz embedding
  $\alpha \colon \lambda' S_n^0 \to Y$.  Viewing such an $\alpha$ as
  map from $t\lambda' S_n^0$ to $t Y$ the composition
  $f \comp \alpha \colon t\lambda' S_n^0 \to X$ is an
  $(LK',C)$-quasi-isometric embedding.  But the minimum distance
  between distinct points in $t\lambda' S_n^0$ is $t\lambda'$ and so
  one can show that $f \comp \alpha$ is a
  $\frac{t\lambda' LK' + C}{t\lambda' - LK'C}$-bilipschitz embedding
  from $t\lambda' S_n^0$.  But
  $\frac{t\lambda' LK' + C}{t\lambda' - LK'C} \to LK' < K$ as
  $\lambda' \to \infty$ so there is a $\Lambda_0$ such that if
  $\lambda' \ge \Lambda_0$ then
  $\frac{t\lambda' LK' + C}{t\lambda' - LK'C} < K$.  So if we had
  $\lambda' \ge \Lambda' = \max\bigl\{\Lambda_0,
  \frac{\Lambda}{t}\bigr\}$ then $f \comp\alpha$ would be a
  $K$-bilipschitz embedding of $\lambda S_n^0$ in $X$ with
  $\lambda = t\lambda' \ge \Lambda$, which would be a contradiction.
  Thus $\Lambda'$ bounds the $\lambda'$ for which there exists a
  $K'$-bilipschitz embedding $\alpha \colon \lambda' S_n^0 \to Y$, as
  required.
\end{proof}

\begin{thm}
  \thmlabel{graph_neg_equiv} Let $\Gamma$ be a graph.  Then the
  following conditions are equivalent.
  \begin{enumerate}
  \item \itmlabel{not_sshortcut_as_graph} $\Gamma$ is not strongly
    shortcut as a graph.
  \item \itmlabel{graph_ngons} $\Gamma$ approximates $n$-gons.
  \end{enumerate}
\end{thm}
\begin{proof}
  $\pitmref{not_sshortcut_as_graph} \Rightarrow \pitmref{graph_ngons}$
  Let $K' > 1$ and let $\alpha \colon S_{n'} \to \Gamma$ be a
  $\frac{1}{K'}$-almost isometric cycle.  Let $n \in \N$ and subdivide
  $S_{n'}$ into $n$ segments of equal length, ignoring the original
  graph structure on $S_{n'}$.  Let $Y \subset S_{n'}$ be the set of
  endpoints of the segments.  Then $Y$ is isometric to $\lambda S_{n}$
  for $\lambda = \frac{n'}{n}$.  Let $\alpha'$ be the composition of
  the inclusion $Y \hookrightarrow S_{n'}$ with $\alpha$.  Let
  $p, q \in Y$ be distinct.  Then $d(p,q) \ge \frac{|S_{n'}|}{n}$ and,
  by \Lemref{global_to_local},
  \begin{align*}
    d\bigl(\alpha'(p),\alpha'(q)\bigr)
    &\ge d(p,q) - \frac{K'-1}{K'} \cdot \frac{|S_{n'}|}{2} \\
    &\ge d(p,q) - \frac{K'-1}{K'} \cdot \frac{n d(p,q)}{2} \\
    &= \Bigl(1 - \frac{n(K'-1)}{2K'} \Bigr)d(p,q)
  \end{align*}
  but $d\bigl(\alpha'(p),\alpha'(q)\bigr) \le d(p,q)$ so, when $K'$ is
  small enough that $\frac{n(K'-1)}{2K'} < 1$, the map $\alpha'$ is
  $K$-bilipschitz for $K = \Bigl(1 - \frac{n(K'-1)}{2K'} \Bigr)^{-1}$.
  Thus, given an arbitrary $n \in \N$ and an $\alpha'$ as above with
  $K'$ small enough, we can obtain a $K$-bilipschitz embedding of
  $\lambda S_n$ in $\Gamma$ with
  $K = \Bigl(1 - \frac{n(K'-1)}{2K'} \Bigr)^{-1}$ and
  $\lambda = \frac{n'}{n}$.  Since $\Gamma$ is not strongly shortcut,
  there exist $\alpha$ as above with $K' > 1$ arbitrarily close to $1$
  and with $n'$ arbitrarily large.  But $K \to 1$ as $K' \to 1$ and
  $\lambda \to \infty$ as $n' \to \infty$ so we have $K$-bilipschitz
  embeddings of $\lambda S_n$ with $K$ arbitrarily close to $1$ and
  $\lambda$ arbitrarily large.

  $\pitmref{graph_ngons} \Rightarrow \pitmref{not_sshortcut_as_graph}$
  Let $n \in \N$, let $K > 1$, let $\lambda > K$ and let
  $\alpha \colon \lambda S_n^0 \to \Gamma$ be a $K$-bilipschitz embedding.
  There is a retraction $r \colon \Gamma \to \Gamma^0$ such that $r$
  is a $(1,1)$-quasi-isometry.  Then the composition $r \comp \alpha$
  is a $(K,1)$-quasi-isometric embedding.  But distinct points in
  $\lambda S_n^0$ are at distance at least $\lambda$ and, since
  $K < \lambda$, this implies that $r \comp \alpha$ is
  $L$-bilipschitz, where $L = \frac{K\lambda + 1}{\lambda - K}$.  View
  $S_n$ as the Cayley graph of $\Z/n\Z$ with generating set $\{1\}$
  and, for $i \in \Z/n\Z$, let $v_i$ be the vertex of $S_n$
  corresponding to $i$.  Then, for each $i$, we have
  $d\bigl(r \comp \alpha(v_i), r \comp \alpha(v_{i+1})\bigr) \le
  \lfloor L\lambda \rfloor$ so there is a combinatorial path
  $\gamma_i \colon P_i \to \Gamma$ of length
  $m_i \in \bigl\{ \lfloor L\lambda \rfloor-1, \lfloor L\lambda
  \rfloor \bigr\}$ from $r \comp \alpha(v_i)$ to
  $r \comp \alpha(v_{i+1})$.  For each $i$, identify the endpoint of
  $P_i$ with the initial point of $P_{i+1}$ to obtain a cycle
  $\gamma \colon S_m \to \Gamma$ with $m = \sum_{i=1}^n m_i$.  Then
  $\alpha$ factors through $\gamma$ via the embedding that sends $v_i$
  to the initial point of $P_i \subset S_m$.  So, viewing $S_n^0$ as a
  subset of $S_m^0$ via this embedding, we have
  $r \comp \alpha(v_i) = \gamma(v_i)$, for each $i$.  Let $x \in S_m$
  and let $v_i$ minimize $d(x,v_i)$.  Then
  $d(x,v_i) \le \frac{\lfloor L\lambda \rfloor}{2}$ and if $\bar x$ is
  the antipode of $x$ and
  $\bar i = i + \bigl\lfloor \frac{n}{2} \bigr\rfloor$ then
  $d(\bar x, v_{\bar i}) \le \lfloor L\lambda \rfloor + \frac{n}{2}$
  so we have the following computation.
  \begin{align*}
    d\bigl(\gamma(x), \gamma(\bar x)\bigr)
    &\ge d\bigl(\gamma(v_i), \gamma(v_{\bar i})\bigr)
      - d\bigl(\gamma(x), \gamma(v_i)\bigr)
      - d\bigl(\gamma(\bar x), \gamma(v_{\bar i})\bigr) \\
    &\ge d\bigl(\gamma(v_i), \gamma(v_{\bar i})\bigr)
      - d(x, v_i) - d(\bar x, v_{\bar i}) \\
    &\ge d\bigl(\gamma(v_i), \gamma(v_{\bar i})\bigr)
      - \frac{\lfloor L\lambda \rfloor}{2}
      - \lfloor L\lambda \rfloor - \frac{n}{2} \\
    &= d\bigl(r \comp \alpha(v_i), r \comp \alpha(v_{\bar i})\bigr)
      - \frac{3\lfloor L\lambda \rfloor}{2} - \frac{n}{2} \\
    &\ge \frac{1}{L} d_{\lambda S_n^0}\bigl(v_i, v_{\bar i})
      - \frac{3\lfloor L\lambda \rfloor}{2} - \frac{n}{2} \\
    &= \frac{\lambda}{L} \Bigl\lfloor \frac{n}{2} \Bigr\rfloor
      - \frac{3\lfloor L\lambda \rfloor}{2}
      - \frac{n}{2} \\
    &\ge \frac{\lambda}{L} \Bigl( \frac{n-1}{2} \Bigr)
      - \frac{3L\lambda}{2} - \frac{n}{2} \\
    &= \frac{1}{m}\Bigl( \frac{\lambda(n-1)}{L}
      - 3L\lambda - n \Bigr)\frac{|S_m|}{2}
  \end{align*}
  Since $m \le nL\lambda$, the above computation implies that
  \[ d\bigl(\gamma(x), \gamma(\bar x)\bigr) \ge \Bigl(
    \frac{n-1}{nL^2} - \frac{3}{n} - \frac{1}{L\lambda}
    \Bigr)\frac{|S_m|}{2} \] so, given $\alpha$ as above, we can
  obtain a $\frac{1}{K'}$-almost isometric cycle in $\Gamma$ of length
  $m$, where
  $\frac{1}{K'} = \bigl( \frac{n-1}{nL^2} - \frac{3}{n} -
  \frac{1}{L\lambda} \bigr)$.  We need only show there exist $\alpha$
  for which $\frac{1}{K'}$ is aritrarily close to $1$ and $m$ is
  arbitrarily large.  By hypothesis, there exist $\alpha$ for which
  $K = \frac{n+1}{n}$ and $\lambda > n$, for arbitrary $n \in \N$.
  But then, as $n \to \infty$, we have
  $m \ge n\bigl(\lfloor L\lambda \rfloor-1\bigr) \to \infty$ and
  $L = \frac{K\lambda + 1}{\lambda - K} \to K$ so
  $\frac{1}{K'} = \bigl( \frac{n-1}{nL^2} - \frac{3}{n} -
  \frac{1}{L\lambda} \bigr) \to 1$.
\end{proof}

Let $X$ be a metric space.  Let $\mathscr{U}$ be a nonprincipal
ultrafilter on $\N$.  Let $(b^{(m)})_{m\in\N}$ be a sequence in $X$.
Let $(s^{(m)})_{m\in\N}$ be a sequence of positive reals such that
$s^{(m)} \to \infty$ as $m \to \infty$.  Consider the set
\[ \mathfrak{X}' = \Biggl\{(x_m)_{m\in\N} \sth
  \text{$\biggl(\frac{d(x_m,b^{(m)})}{s^{(m)}}\biggr)_{m\in\N}$ is
    bounded} \Biggr\} \] of \defterm{sequences in $X$ that are bounded
  with respect to} the \defterm{basepoint sequence} $(b^{(m)})_m$ and
the \defterm{scaling sequence} $(s^{(m)})_m$.  For
$(x_m)_m, (x'_m)_m \in \mathfrak{X}'$,
\[ \bar d\bigl((x_m)_m,(x'_m)_m\bigr) = \lim_{\mathscr{U}}
  \frac{d(x_m,x'_m)}{s^{(m)}} \] defines a pseudometric on
$\mathfrak{X}'$.  The \defterm{asymptotic cone} $\mathfrak{X}$ of $X$
with respect to the nonprincipal ultrafilter $\mathscr{U}$, the
\defterm{basepoint sequence} $(b^{(m)})_m$ and the \defterm{scaling
  sequence} $(s^{(m)})_m$ is the metric space obtained from
$\mathfrak{X}'$ and $\bar d$ by identifying $(x_m)_m$ and $(x'_m)_m$
whenever $\bar d\bigl((x_m)_m,(x'_m)_m\bigr) = 0$.

Note that \Thmref{general_metric_neg_equiv} and
\Corref{as_cones_general_ss} apply to general metric spaces and not
just rough geodesic metric spaces.

\begin{thm}
  \thmlabel{general_metric_neg_equiv} Let $X$ be a metric space.  Then
  the following conditions are equivalent.
  \begin{enumerate}
  \item \itmlabel{general_asymptotic_circle} There is an asymptotic
    cone of $X$ that contains an isometric copy of the Riemannian
    circle of unit length.
  \item \itmlabel{general_ngons} $X$ approximates $n$-gons.
  \end{enumerate}
\end{thm}
\begin{proof}
  $\pitmref{general_asymptotic_circle} \Rightarrow
  \pitmref{general_ngons}$ Suppose $S \subset \mathcal{X}$ is a
  subspace isometric to the Riemannian circle of unit length in the
  \defterm{asymptotic cone} $\mathcal{X}$ of $X$ with respect to a
  nonprincipal ultrafilter $\mathscr{U}$, a \defterm{basepoint
    sequence} $(b^{(m)})_m$ and a \defterm{scaling sequence}
  $(s^{(m)})_m$.  Take any $n \in \N$, any $K > 1$ and any
  $\Lambda > 0$.  We will construct a $K$-bilipschitz map
  $\alpha \colon \lambda S_n^0 \to X$ with $\lambda \ge \Lambda$.
  Subdivide $S$ into $n$ segments of equal length and let $S^0$ denote
  the set of endpoints of the segments.  For each $\epsilon > 0$ and
  each $p,q \in S^0$ reprented by $(p_m)_m$ and $(q_m)_m$, there is an
  $A_{\epsilon}^{p,q} \in \mathscr{U}$ such that
  \[ d(p,q) - \epsilon \le \frac{d(p_m,q_m)}{s^{(m)}} \le d(p,q) +
    \epsilon \] for all $m \in A_{\epsilon}^{p,q}$.  There are
  finitely many pairs $p,q \in S^0$ so
  $A_{\epsilon} = \bigcap_{p,q} A_{\epsilon}^{p,q} \in \mathscr{U}$.
  Then, for any distinct $p,q \in S^0$,
  \[ \Bigl(1 - \frac{\epsilon}{d(p,q)}\Bigr)s^{(m)}d(p,q) \le
    d(p_m,q_m) \le \Bigl(1 +
    \frac{\epsilon}{d(p,q)}\Bigr)s^{(m)}d(p,q) \] for all
  $m \in A_{\epsilon}$.  But $d(p,q) \ge \frac{1}{n}$ and so 
  \[ (1 - n\epsilon)s^{(m)}d(p,q) \le d(p_m,q_m) \le (1 +
    n\epsilon)s^{(m)}d(p,q) \] for all $m \in A_{\epsilon}$.  So if
  $n\epsilon < 1$ then, for $m \in A_{\epsilon}$, the map
  \begin{align*}
    \alpha_m \colon s^{(m)} S^0 &\to X \\
    p &\mapsto p_m
  \end{align*}
  is bilipschitz with bilipschitz constant
  $\max\bigl\{1 + n\epsilon, \frac{1}{1-n\epsilon} \bigr\} =
  \frac{1}{1-n\epsilon}$.  The space $s^{(m)} S^0$ is isometric to
  $\frac{s^{(m)}}{n} S_n^0$ so if we chose $\epsilon$ small enough so
  that $n\epsilon < 1$ and $\frac{1}{1-n\epsilon} < K$ and we take
  $m \in A_{\epsilon}$ large enough that
  $\frac{s^{(m)}}{n} \ge \Lambda$ then we can take
  $\alpha = \alpha_m$.
  
  $\pitmref{general_ngons} \Rightarrow
  \pitmref{general_asymptotic_circle}$ For $m \in \N$, there exists a
  $\frac{m+1}{m}$-bilipschitz map
  $\alpha_m \colon \lambda_m S_{2^m}^0 \to X$ with $\lambda_m \ge m$.
  Metrize the group
  $\frac{1}{2^m}\Z = \bigl\{\frac{k}{2^m} \sth k \in \Z\bigr\} \subset
  \R$ with the subspace metric and metrize the quotient group
  $\frac{1}{2^m}\Z/\Z$ with the quotient metric.  Then
  $\frac{1}{2^m} S_{2^m}^0$ is isometric to $\frac{1}{2^m}\Z/\Z$.  Via
  this isometry we identify the vertex set $S_{2^m}^0$ with the
  elements of $\frac{1}{2^m}\Z/\Z$.  Thus we view
  $\frac{1}{2^m} S_{2^m}^0$ as a metric subspace of the Riemannian
  circle of unit length $S = \R/\Z$.  By this identification, the
  union $S_{\mathbb{D}} = \bigcup_{m \in \N} S_{2^m}^0 \subset S$ is
  the \defterm{dyadic circle} $\Z\bigl[\frac{1}{2}\bigr]/\Z$.  The
  dyadic circle $S_{\mathbb{D}}$ is dense in $S$.  Thus, since
  asymptotic cones are complete metric spaces
  \cite[Proposition~10.70]{Drutu:2018}, it will suffice to
  isometrically embed $S_{\mathbb{D}}$ into an asymptotic cone of $X$.

  View $\alpha_m$ as an $\frac{m+1}{m}$-bilipschitz map from
  $\frac{1}{2^m} S_{2^m}^0$ to $\frac{1}{\lambda_m 2^m}X$.  Set
  $b^{(m)} = \alpha_m(0)$ and set $s^{(m)} = \lambda_m 2^m$.  Every
  nonzero element of $S_{\mathbb{D}}$ can be uniquely represented as
  $\frac{k}{2^{\ell}}$ with $k$ odd and satisfying
  $0 \le k < 2^{\ell}$.  For any such representation
  $\frac{k}{2^{\ell}}$ and any $m \in \N$, set
  \[ x_{\frac{k}{2^{\ell}}}^{(m)} =
    \begin{cases}
      b^{(m)} &\text{if $m < \ell$} \\
      \alpha_m\bigl(\frac{k}{2^{\ell}}\bigr) &\text{if $m \ge \ell$}
    \end{cases} \] and set $x_0^{(m)} = b^{(m)}$.  Then, for any non
  principal ultrafilter $\mathscr{U}$, the expression
  $p \mapsto (x_p^{(m)})_m$ defines an isometric embedding of
  $S_{\mathbb{D}}$ into the asymptotic cone $\mathcal{X}$ of $X$ with
  respect to $\mathscr{U}$, the basepoint sequence $(b^{(m)})_m$ and
  the scaling sequence $S^{(m)}$.  Indeed, for every
  $p,q \in S_{\mathbb{D}}$,
  \[ \frac{d(x_p^{(m)},x_q^{(m)})}{s^{(m)}} =
    \frac{d\bigl(\alpha_m(p),\alpha_m(q)\bigr)}{\lambda_m 2^m} \le
    \frac{\frac{m+1}{m} \cdot \lambda_m d_{S_{2^m}}(p,q)}{\lambda_m
      2^m} = \frac{m+1}{m} \cdot d_{S_{\mathbb{D}}}(p,q) \]
  and
  \[ \frac{d(x_p^{(m)},x_q^{(m)})}{s^{(m)}} =
    \frac{d\bigl(\alpha_m(p),\alpha_m(q)\bigr)}{\lambda_m 2^m} \ge
    \frac{\frac{m}{m+1} \cdot \lambda_m d_{S_{2^m}}(p,q)}{\lambda_m
      2^m} = \frac{m}{m+1} \cdot d_{S_{\mathbb{D}}}(p,q) \] whenever
  $m$ is large enough.
\end{proof}

\begin{cor}
  \corlabel{as_cones_general_ss} Let $X$ be a metric space and let
  $\mathfrak{X}$ be an asymptotic cone of $X$.  Suppose that $X$ does
  not approximate $n$-gons.  Then $\mathfrak{X}$ does not approximate
  $n$-gons.
\end{cor}
\begin{proof}
  By \Thmref{general_metric_neg_equiv}, it suffices to show that any
  asymptotic cone $\mathfrak{X}'$ of $\mathfrak{X}$ does not contain
  an isometric copy of a Riemannian circle of unit length.  But
  $\mathfrak{X}'$ is isometric to an asymptotic cone of $X$
  \cite[Corollary~10.80]{Drutu:2018} so does not contain an isometric
  copy of a Riemannian circle of unit length by
  \Thmref{general_metric_neg_equiv}.
\end{proof}

\begin{thm}
  \thmlabel{not_sshortcut_equiv} Let $X$ be an $R$-rough geodesic
  metric space.  The following conditions are equivalent.
  \begin{enumerate}
  \item \itmlabel{not_sshortcut} $X$ is not strongly shortcut.
  \item \itmlabel{k_2r_quasi_circles} For every $L > 1$ there exist
    $(L,4R)$-quasi-isometric embeddings of arbitrarily long
    Riemannian circles in $X$.
  \item \itmlabel{quasi_circles} For every $L > 1$ there is a
    $C \ge 0$ such that there exist $(L,C)$-quasi-isometric embeddings
    of arbitrarily long Riemannian circles in $X$.
  \item \itmlabel{ngons} $X$ approximates $n$-gons.
  \item \itmlabel{asymptotic_circle} There is an asymptotic cone of
    $X$ that contains an isometric copy of the Riemannian circle of
    unit length.
  \end{enumerate}
\end{thm}
\begin{proof}
  Conditions \pitmref{ngons} and \pitmref{asymptotic_circle} are
  equivalent for general metric spaces, by
  \Thmref{general_metric_neg_equiv}.  So it will suffice to prove the
  equivalence of conditions \pitmref{not_sshortcut},
  \pitmref{k_2r_quasi_circles}, \pitmref{quasi_circles} and
  \pitmref{ngons}.
  
  $\pitmref{not_sshortcut} \Rightarrow \pitmref{k_2r_quasi_circles}$
  Let $N = 2$, let $L > 1$ be arbitrary, let $K > 1$ be small enough
  to satisfy \Lemref{circle_tightening} and let
  $\alpha \colon S \to X$ be a $\frac{1}{K}$-almost isometric
  $R$-circle in $X$ with $|S|$ arbitrarily larger than the $M$ from
  \Lemref{circle_tightening}.  Then the limit $R$-circle
  $\alpha_{\infty} \colon S_{\infty} \to X$ given by
  \Lemref{circle_tightening} is an $(L,4R)$-quasi-isometric embedding
  of a Riemannian circle of length at least $\frac{|S|}{2}$.

  $\pitmref{k_2r_quasi_circles} \Rightarrow
  \pitmref{quasi_circles}$ This is immediate.

  $\pitmref{quasi_circles} \Rightarrow \pitmref{ngons}$ Let
  $\alpha \colon S \to X$ be $(L,C)$-quasi-isometric embedding of a
  Riemannian circle.  Let $n \in \N$, subdivide $S$ into $n$ segments
  of equal length and let $Y$ be the set of endpoints of the segments.
  Then $Y$ is isometric to $\frac{|S|}{n} S_n$ and
  \[\Bigl(\frac{1}{L} - \frac{nC}{|S|} \Bigr)d(p,q)
    \le d\bigl(\alpha(p),\alpha(q)\bigr) \le \Bigl(L +
    \frac{nC}{|S|}\Bigr)d(p,q) \] for distinct $p,q \in Y$.  By
  hypothesis, there exist arbitrarily long $\alpha$ with $L$
  arbitrarily close to $1$ and so $\alpha|_Y$ is a $K$-bilipschitz
  embedding of $\lambda S_n$ for $\lambda = \frac{|S|}{n}$ arbitrarily
  large and
  $K = \max\Bigl\{\bigl(\frac{1}{L} - \frac{nC}{|S|}\bigr)^{-1}, L +
  \frac{nC}{|S|}\Bigr\}$ arbitrarily close to $1$.

  $\pitmref{ngons} \Rightarrow \pitmref{not_sshortcut}$ Let
  $n \in \N$, let $L > 1$, let $\lambda > 0$ and let
  $\alpha \colon \lambda S_n^0 \to X$ be an $L$-bilipschitz embedding.
  View $S_n$ as the Cayley graph of $\Z/n\Z$ with generating set
  $\{1\}$ and, for $i \in \Z/n\Z$, let $v_i$ be the vertex of $S_n$
  corresponding to $i$.  Then, for each $i$, we have
  $d\bigl(\alpha(v_i), \alpha(v_{i+1})\bigr) \le L\lambda$ so, by
  scaling an $R$-rough geodesic, there is an $R$-path
  $\gamma'_i \colon P_i \to X$ of length $|P_i| = L\lambda$ from
  $\alpha(v_i)$ to $\alpha(v_{i+1})$.  Let $\gamma_i$ be the
  concatenation $c_i\gamma'_i$ where $c_i\colon [0,R] \to X$ is the
  constant path of length $R$ at $\alpha(v_i)$.  For each $i$,
  identify the endpoint of $P_i$ with the initial point of $P_{i+1}$
  to obtain an $R$-circle $\gamma \colon S \to X$ with
  $|S| = n(L\lambda + R)$.  Then $\alpha$ factors through $\gamma$ via
  the embedding that sends $v_i$ to the initial point of
  $P_i \subset S$.  So, viewing $S_n^0$ as a subset of $S$ via this
  embedding, we have $\alpha(v_i) = \gamma(v_i)$, for each $i$.  Let
  $x \in S$ and let $v_i$ minimize $d(x,v_i)$.  Then
  $d(x,v_i) \le \frac{L\lambda + R}{2}$ and if $\bar x$ is the
  antipode of $x$ and
  $\bar i = i + \bigl\lfloor \frac{n}{2} \bigr\rfloor$ then
  $d(\bar x, v_{\bar i}) \le L\lambda + R$ so we have the following
  computation.
  \begin{align*}
    d\bigl(\gamma(x), \gamma(\bar x)\bigr)
    &\ge d\bigl(\gamma(v_i), \gamma(v_{\bar i})\bigr)
      - d\bigl(\gamma(x), \gamma(v_i)\bigr) 
      - d\bigl(\gamma(\bar x), \gamma(v_{\bar i})\bigr) \\
    &\ge d\bigl(\gamma(v_i), \gamma(v_{\bar i})\bigr)
      - d(x, v_i) - R - d(\bar x, v_{\bar i}) - R \\
    &\ge d\bigl(\gamma(v_i), \gamma(v_{\bar i})\bigr)
      - \frac{L\lambda + R}{2}
      - (L\lambda+R) - 2R \\
    &= d\bigl(\alpha(v_i), \alpha(v_{\bar i})\bigr)
      - \frac{3L\lambda + 7R}{2} \\
    &\ge \frac{1}{L} d_{\lambda S_n^0}\bigl(v_i, v_{\bar i})
      - \frac{3L\lambda + 7R}{2} \\
    &= \frac{\lambda}{L} \Bigl\lfloor \frac{n}{2} \Bigr\rfloor
      - \frac{3L\lambda + 7R}{2} \\
    &\ge \frac{\lambda}{L} \Bigl( \frac{n-1}{2} \Bigr)
      - \frac{3L\lambda + 7R}{2} \\
    &= \frac{1}{|S|}\Bigl( \frac{\lambda(n-1)}{L}
      - 3L\lambda - 7R \Bigr)\frac{|S|}{2} \\
    &= \frac{1}{n(L\lambda + R)}\Bigl( \frac{\lambda(n-1)}{L}
      - 3L\lambda - 7R \Bigr)\frac{|S|}{2} \\
    &= \biggl( \frac{\lambda n-\lambda}{L^2\lambda n + nLR}
      - \frac{3L\lambda}{L\lambda n + nR}
      - \frac{7R}{L\lambda n + nR} \biggr)\frac{|S|}{2}
  \end{align*}
  So, given $\alpha$ as above, we can obtain a $\frac{1}{K}$-almost
  isometric $R$-circle in $X$ of length $|S| = n(L\lambda + R)$, where
  $\frac{1}{K} = \Bigl( \frac{\lambda n-\lambda}{L^2\lambda n + nLR}
  - \frac{3L\lambda}{L\lambda n + nR} - \frac{7R}{L\lambda n + nR}
  \Bigr)$.  We need only show there exist $\alpha$ for which
  $\frac{1}{K}$ is aritrarily close to $1$ and $|S|$ is arbitrarily
  large.  By hypothesis, there exist $\alpha$ for which
  $L = \frac{n+1}{n}$ and $\lambda > n$, for arbitrary $n \in \N$.
  But then, as $n \to \infty$, we have
  $|S| = n(L\lambda + R) \to \infty$ and
  $\frac{1}{K} = \Bigl( \frac{\lambda n-\lambda}{L^2\lambda n + nLR}
  - \frac{3L\lambda}{L\lambda n + nR} - \frac{7R}{L\lambda n + nR}
  \Bigr) \to 1$.
\end{proof}

\begin{cor}
  \corlabel{as_cones_ss} Let $X$ be a metric space.  If $X$ is
  strongly shortcut then every asymptotic cone of $X$ is strongly
  shortcut.
\end{cor}
\begin{proof}
  Follows immediately from \Thmref{not_sshortcut_equiv} and
  \Corref{as_cones_general_ss}.
\end{proof}

\begin{cor}
  \corlabel{rough_approx_inv} The strong shortcut property is a rough
  approximability invariant of rough geodesic metric spaces.  In
  particular, the strong shortcut property is a rough similarity
  invariant of rough geodesic metric spaces.
\end{cor}
\begin{proof}
  Follows immediately from \Thmref{not_sshortcut_equiv} and
  \Propref{not_approx_ngon_invariant}.
  %
\end{proof}

\subsection{Instability under quasi-isometries}
\sseclabel{not_qi_inv}

\begin{figure}[ht]
  \newcommand{\slen}{2mm}
  \newcommand{\pos}[1]{#1*(#1+1)*0.5}
  \centering
  \begin{tikzpicture}
    \draw[step=\slen,rotate=-45]
    ($\pos{0}*(\slen,\slen)$)
    grid ($\pos{1}*(\slen,\slen)$)
    grid ($\pos{2}*(\slen,\slen)$)
    grid ($\pos{3}*(\slen,\slen)$)
    grid ($\pos{4}*(\slen,\slen)$)
    grid ($\pos{5}*(\slen,\slen)$)
    grid ($\pos{6}*(\slen,\slen)$)
    grid ($\pos{7}*(\slen,\slen)$);

    \node at ($1.4142*\pos{7}*(\slen,0)+(1em,0)$) {$\ldots$};
  \end{tikzpicture}
  \caption{Continuing the pattern, one obtains an infinite graph that
    is strongly shortcut because it is the $1$-skeleton of a
    finite-dimensional $\CAT(0)$ cube complex.  Subdividing the
    \emph{interior} edges of each $n \times n$ grid results in a
    quasi-isometric graph that is not strongly shortcut.}
  \figlabel{gridseq}
\end{figure}

In light of \Corref{rough_approx_inv}, we should point out that the
strong shortcut property is not a quasi-isometry invariant.  The
$1$-skeleton of an $n \times n$ grid of squares is strongly shortcut
but subdividing its \emph{interior} edges causes its boundary cycle to
become isometrically embedded.  We can construct a strongly shortcut
graph $\Gamma$ that contains isometric copies of $1$-skeletons of
larger and larger $n \times n$ grids.  See \Figref{gridseq}.
Subdividing the interior edges of each $n \times n$ grid of $\Gamma$
does not change the quasi-isometry type but results in a graph that is
not strongly shortcut because it contains arbitrarily long
isometrically embedded cycles.

\section{The Circle Tightening Lemma}
\seclabel{circle_tightening}

The Circle Tightening Lemma describes how one may perform surgery on
an almost isometric $R$-circle to obtain a quasi-isometrically
embedded $R$-circle, assuming the various constants are chosen
appropriately.  A version of this lemma first appeared implicitly in
the proof of a proposition in an earlier work of this author
\cite[Proposition~3.5]{Hoda:shortcut_graphs:2022} where it applied
only to graphs.  Here we state and prove a generalization to (rough)
geodesic metric spaces.

\subsection{Tightening sequence for a Riemannian circle}
\sseclabel{tightening_riemannian_circle}

\begin{figure}[ht]
  \newcommand{\circrad}{1.2cm}
  \newcommand{\seglen}{2cm}

  \centering
  
  \[ \begin{tikzcd}
      &
      \begin{tikzpicture}
        \draw (0,0) -- (\seglen,0);
        \node at (\seglen/2,-3ex) {$P_i$};
        \node[vertex] at (0, 0) {};
        \node[vertex] at (\seglen, 0) {};
      \end{tikzpicture}
      \ar[dl,hook',start anchor={south west}]
      \ar[dr,start anchor={south east}]
      & \\
      \begin{tikzpicture}
        \draw (10:\circrad) arc[start angle=10, end angle=300, radius=\circrad];
        \draw[train] (300:\circrad) arc[start angle=300, end angle=370, radius=\circrad];
        \node at (0,0) {$S_i$};
        \node at (155:\circrad + 1em) {$P_i$};
        \node at (335:\circrad + 1em) {$Q_i$};
        \node[vertex] at (10:\circrad) {};
        \node[vertex] at (300:\circrad) {};
      \end{tikzpicture}
      & &
      \begin{tikzpicture}
        \draw (10:\circrad) arc[start angle=10, end angle=300, radius=\circrad];
        \draw[red] (300:\circrad) -- (10:\circrad);
        \node at (0,0) {$S_{i+1}$};
        \node at (155:\circrad + 1.5em) {$\im(P_i)$};
        \node at (335:\circrad + 0.5em) {$\bar Q_i$};
        \node[vertex] at (10:\circrad) {};
        \node[vertex] at (300:\circrad) {};
      \end{tikzpicture}
      \\
      &
      \begin{tikzpicture}
        \draw (0,0) -- (\seglen,0);
        \node at (\seglen/2,2ex) {$\pc{i}$};
        \node[overtex] at (0, 0) {};
        \node[overtex] at (\seglen, 0) {};
      \end{tikzpicture}
      \ar[ul,hook',start anchor={north west}]
      \ar[uu,hook]
      \ar[ur,hook,start anchor={north east}]
      &
  \end{tikzcd} \]
  
\caption{In a circle tightening sequence, the circle $S_{i+1}$ is
  either equal to $S_i$ or is obtained from $S_i$ by replacing some
  geodesic segment $Q_i$ of $S_i$ with a shorter segment $\bar Q_i$,
  possibly of zero length.}
  \figlabel{circle_tightening}
\end{figure}

Let $S$ be a Riemannian circle.  A \defterm{tightening sequence} for
$S$ is a sequence of intervals and Riemannian circles $(P_i)_i$, a
sequence of Riemannian circles $(S_i)_i$ and sequences of maps
\[ \begin{tikzcd}
    S = S_0  & P_0 \ar[l,hook'] \ar[r] & S_1  & P_1 \ar[l,hook'] \ar[r] & S_2 & P_2 \ar[l,hook'] \ar[r] & \cdots
  \end{tikzcd} \] such that, for each $i$, either
\begin{enumerate}
\item
  \begin{enumerate}
  \item $P_i \hookrightarrow S_i$ and $P_i \to S_{i+1}$ are continuous
    paths of unit speed, and
  \item $P_i \to S_{i+1}$ is injective on the interior $\pc{i}$ of
    $P_i$, and
  \item $|P_i| \ge \frac{|S|}{2}$, and
  \item $|S_{i+1}| < |S_i|$; or
  \end{enumerate}
\item
  \begin{enumerate}
  \item $\pc{i} = P_i = S_i = S_{i+1}$, and
  \item $P_i \hookrightarrow S_i$ and $P_i \to S_{i+1}$ are identity
    maps.
  \end{enumerate}
\end{enumerate}
See \Figref{circle_tightening}.  Then, for each $i$, the circle
$S_{i+1}$ is obtained from $S_i$ by replacing
$Q_i = S \setminus \pc{i}$ with $\bar Q_i$, where either
$Q_i = \bar Q_i = \emptyset$ or $Q_i$ and $\bar Q_i$ are intervals
with $|\bar Q_i| < |Q_i|$.  So we also have a commutative diagram of
$1$-Lipschitz maps
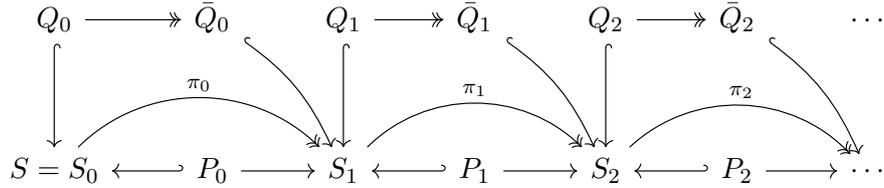
\[ \begin{tikzcd}
    Q_0 \ar[r,two heads] \ar[dd,hook] & \bar Q_0 \ar[ddr,bend left=15,hook] & Q_1 \ar[r,two heads] \ar[dd,hook] & \bar Q_1 \ar[ddr,bend left=15,hook] & Q_2 \ar[r,two heads] \ar[dd,hook] & \bar Q_2 \ar[ddr,bend left=15,hook] & \cdots \\
    \\
    S=S_0 \ar[rr,bend left=45,two heads,"\pi_0"] & P_0 \ar[l,hook'] \ar[r] & S_1 \ar[rr,bend left=45,two heads,"\pi_1"] & P_1 \ar[l,hook'] \ar[r] & S_2  \ar[rr,bend left=45,two heads,"\pi_2"] & P_2 \ar[l,hook'] \ar[r] & \cdots
  \end{tikzcd} \] where each $Q_i \twoheadrightarrow \bar Q_i$ is
affine.  We call the $Q_i$ the \defterm{tightened segments} of the
tightening sequence.  We let $\pi^{(i)}$ denote the composition
$\pi_{i-1} \comp \pi_{i-2} \comp \cdots \comp \pi_0$.  A tightening
sequence is \defterm{eventually constant} if $S_i = S_{i+1}$ for all
large enough $i$.

Let $\pc{0,j}$ be the limit of the diagram
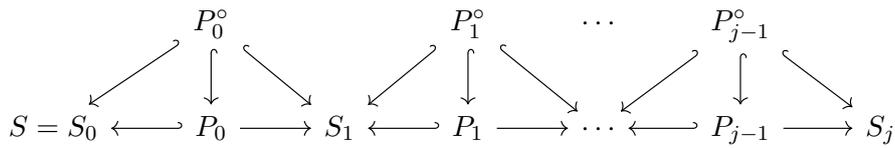
\[ \begin{tikzcd}
    & \pc{0} \ar[d,hook] \ar[dl,hook'] \ar[dr,hook] & & \pc{1} \ar[d,hook] \ar[dl,hook'] \ar[dr,hook] & \cdots & \pc{j-1} \ar[d,hook] \ar[dl,hook'] \ar[dr,hook] \\
    S=S_0 & P_0 \ar[l,hook'] \ar[r] & S_1 & P_1 \ar[l,hook'] \ar[r] & \cdots & P_{j-1} \ar[l,hook'] \ar[r] & S_j
  \end{tikzcd} \]
in the category of topological spaces and continuous maps.
Concretely, we have $\pc{0,0} = S_0$ and $\pc{0,1} = \pc{0}$ and
$\pc{0,j} = \pc{0,j-1} \cap \pc{j-1}$ where the intersection is
taken in $S_{j-1}$.  Thus we have the following commutative
diagram.
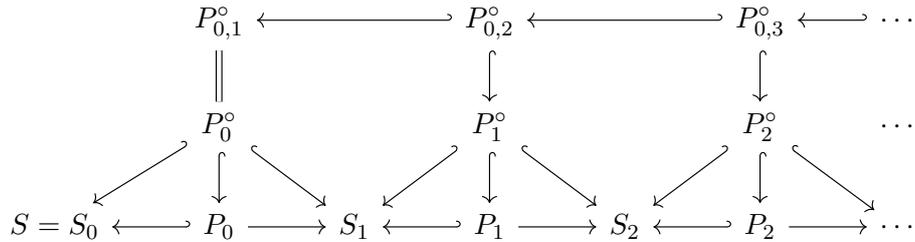
\[ \begin{tikzcd}
    & \pc{0,1} \ar[d,equal] && \pc{0,2} \ar[d,hook] \ar[ll,hook'] && \pc{0,3} \ar[d,hook] \ar[ll,hook'] & \cdots \ar[l,hook'] \\
    & \pc{0} \ar[d,hook] \ar[dl,hook'] \ar[dr,hook] & & \pc{1} \ar[d,hook] \ar[dl,hook'] \ar[dr,hook] & & \pc{2} \ar[d,hook] \ar[dl,hook'] \ar[dr,hook] & \cdots \\
    S=S_0 & P_0 \ar[l,hook'] \ar[r] & S_1 & P_1 \ar[l,hook'] \ar[r] & S_2 & P_2 \ar[l,hook'] \ar[r] & \cdots
  \end{tikzcd} \] We can think of $\pc{0,j}$ as the original points of
$S$ that are not replaced until at least step $j$ of the
``construction'' of the $S_i$, where the $j$th step of the
construction refers to the operation of replacing $Q_j$ with
$\bar Q_j$ in order to obtain $S_{j+1}$ from $S_j$.

\begin{figure}[ht]
  \newcommand{\circrad}{2cm}

  \newcommand{\arc}[2]{
    \draw (#1:\circrad) arc[start angle=#1, end angle=#2, radius=\circrad]}
  \newcommand{\replarc}[4]{
    \draw[train] (#1:\circrad) arc[start angle=#1, end angle=#2, radius=\circrad];
    \draw (#1:\circrad) -- (#2:\circrad)
    coordinate[midway](m);
    \node at (#1/2+#2/2:\circrad+1ex+1em) {#3};
    \path (m) +(180+#1/2+#2/2:1em) coordinate (M);
    \node at (M) {#4}}

  \centering
  \begin{tikzpicture}
    \replarc{-75}{-20}{$Q_1$}{};
    \path (-75:\circrad) -- (-20:\circrad) coordinate[midway] (A);
    \arc{-20}{-5};
    \replarc{-5}{55}{$Q_0$}{};
    \path (-5:\circrad) -- (55:\circrad) coordinate[midway] (B);
    \arc{55}{65};

    \draw(A) -- (B) coordinate[midway] (C);
    \path (C) +(-1em,0) coordinate (L);
    \node at (L) {$\bar Q_4$};
    \draw[cyan,train] (A) -- (-20:\circrad) arc[start angle=-20, end angle=-5, radius=\circrad] -- (B);
    
    \replarc{65}{95}{$Q_2$}{$\bar Q_2$};
    \arc{95}{110};
    \replarc{110}{135}{$Q_5$}{$\bar Q_5$};
    \arc{135}{175};
    \replarc{175}{205}{$Q_3$}{$\bar Q_3$};
    \arc{205}{245};
    \replarc{245}{275}{$Q_6$}{$\bar Q_6$};
    \arc{275}{285};
  \end{tikzpicture}
  
  \caption{A circle tightening sequence that is disjoint up to $4$ but
    not disjoint up to $5$.  The outer circle is the initial circle
    $S = S_0$.  For $i \ge 0$, the circle $S_{i+1}$ is obtained from $S_i$
    by replacing the geodesic segment $Q_i \subset S_{i}$ (indicated
    by perpendicular markings) with a shorter sequence $\bar Q_i$.  The
    segment $Q_4$ (drawn in cyan) is the first replaced segment that
    cannot be viewed as a subspace of $S$ since it is not contained
    in $\pc{0,4}$, which can be viewed as
    $S \setminus \bigcup_{i=0}^3 Q_i$.}
  \figlabel{circle_tightening_disjoint}
\end{figure}

We say that a tightening sequence is \defterm{disjoint up to $j$} if
$Q_i \subset \pc{0,i}$ for all $i < j$, where $\pc{0,i}$ is viewed
as a subspace of $S_i$ via the embedding
$\pc{0,i} \hookrightarrow S_i$.  See
\Figref{circle_tightening_disjoint}.  We say that a tightening
sequence is \defterm{completely disjoint} if it is disjoint up to
$j$ for every $j$.

If a tightening sequence is disjoint up to $j$, for $i<j$, we have
$Q_i \sqcup \pc{0,i+1} = \pc{0,i} \hookrightarrow S$.  So, for
$i < j$, we may think of the $Q_i$ as disjoint subspaces of $S$ with
$S \setminus \bigcup_{i=0}^{j-1}Q_i = \pc{0,j}$ in $S$.  Since $S_j$
is obtained from $S$ by replacing $Q_i$ with $\bar Q_i$, for each
$i < j$, we see then that the $\bar Q_i$, with $i < j$, embed
disjointly in $S_j$ with
$\pc{0,j} = S_j \setminus \bigcup_{i=0}^{j-1}\bar Q_i$ in $S_j$.

If a tightening sequence is completely disjoint then the $Q_i$ all
embed disjointly in $S$ and the complement of their union in $S$ is
$\pc{0,\infty} = \bigcap_{j=1}^{\infty} \pc{0,j}$.

\begin{lem}
  \lemlabel{lim_riemannian_circle} Consider a tightening sequence for
  a Riemannian circle $S$ with the same notation as above.  If the
  tightening sequence is completely disjoint and the sum
  $\sum_{i=0}^{\infty} |Q_i|$ of the tightened segment lengths is
  strictly less than $|S|$ then
  \begin{align*}
    \delta \colon S \times S &\to \R_{\ge 0} \\
    (x,y) &\mapsto \lim_{i \to \infty} d_{S_i}\bigl(\pi^{(i)}(x),
    \pi^{(i)}(y) \bigr)
  \end{align*} defines a pseudometric on $S$ such that
  the induced metric quotient $S_{\infty}$ of $(S,\delta)$,
  called the \defterm{limit Riemannian circle} of the tightening
  sequence, is a Riemannian circle of length
  $\lim_{i \to \infty} |S_i|$.
\end{lem}
\begin{proof}
  The $\pi^{(i)}$ are isomorphisms on fundamental group so, for each
  $i$, we have a commuting diagram \[
    \begin{tikzcd}
      \R \ar[d] \ar[r,"\tilde \pi_i"]  & \R \ar[d] \\
      S_i \ar[r, "\pi_i"] & S_{i+1}
    \end{tikzcd} \] where $\R \to S_i$ and $\R \to S_{i+1}$ are the
  quotient maps from $(\R,+)$ by the subgroups $|S_i|\Z$ and
  $|S_{i+1}|\Z$, respectively.  Then, since $\pi_i$ is $1$-Lipzchitz,
  so is $\tilde \pi_i$.  Without loss of generality, the map
  $\tilde \pi_i$ sends $0$ to $0$ and preserves order, in the sense
  that $s \le r$ implies $\tilde \pi_i(s) \le \tilde \pi_i(r)$.  Let
  $\tilde \pi^{(i)} \colon \R \to \R$ be the composition
  $\tilde \pi_0 \comp \tilde \pi_1 \comp \cdots \comp \tilde
  \pi_{i-1}$ so that the diagram \[ \begin{tikzcd}
      \R \ar[d] \ar[r,"\tilde \pi^{(i)}"]  & \R \ar[d] \\
      S \ar[r, "\pi^{(i)}"] & S_i
    \end{tikzcd} \] commutes and satisfies the same properties as the
  previous diagram.  Then for $r \in \R$, the sequence
  $\bigl(\pi^{(i)}(r)\bigr)_i$ is either nonnegative and nonincreasing
  or nonpositive and nondecreasing.  In either case the limit exists
  so we can define a limit function
  \begin{align*}
    \pi^{(\infty)} \colon \R &\to \R \\
    r &\mapsto  \lim_{i \to \infty} \pi^{(i)}(r)
  \end{align*}
  which is also $1$-Lipschitz, sends $0$ to $0$ and preserves order.
  
  By assumption $\sum_{i=1}^{\infty}|Q_i| < |S|$ so
  $\pi^{(i)}\bigl(|S|\bigr) = |S_i| \ge |S| - \sum_{i=1}^{\infty}|Q_i|
  > 0$ and so we have the following.
  \[\pi^{(\infty)}\bigl(|S|\bigr) = \lim_{i \to \infty}|S_i| > 0 \]
  For $r \in \R$, we have
  $\pi^{(i)}\bigl(r + |S|\bigr) = \pi^{(i)}(r) + |S_i|$ so
  \[ \pi^{(\infty)}\bigl(r + |S|\bigr) = \pi^{(\infty)}(r) +
    \lim_{i\to\infty}|S_i| \] which implies that if $\R \to \bar S$ is
  the quotient map of $(\R,+)$ with kernel
  $\bigl(\lim_{i \to \infty}|S_i|\bigr)\Z$ then the map
  \begin{align*}
    \pi^{(\infty)} \colon S &\to S_{\infty} \\
    r + |S|\Z &\mapsto \pi^{(\infty)}(r) + \bigl(\lim_{i \to \infty}|S_i|\bigr)\Z
  \end{align*}
  is well defined and makes the diagram \[
    \begin{tikzcd}
      \R \ar[d] \ar[r,"\tilde \pi^{(\infty)}"]  & \R \ar[d] \\
      S \ar[r, "\pi^{(\infty)}"] & \bar S
    \end{tikzcd} \]
  commute.

  Then $|\bar S| = \lim_{i \to \infty}|S_i|$ and, for $r+|S|\Z$ and
  $s+|S|\Z$ in $S$, we have
  \begin{align*}
    &d_{\bar S}\Bigl(\pi^{(\infty)}\bigl(r+|S|\Z\bigr),
    \pi^{(\infty)}\bigl(s+|S|\Z\bigr)\Bigr) \\
    &= d_{\bar S}\Bigl(\pi^{(\infty)}(r) + |\bar S|\Z,
      \pi^{(\infty)}(s) + |\bar S|\Z\Bigr) \\
    &= \min_{k \in \Z}\Bigl|\pi^{(\infty)}(r) -
      \pi^{(\infty)}(s) + |\bar S|k \Bigr| \\
    &= \min_{k \in I}\Bigl|\pi^{(\infty)}(r) -
      \pi^{(\infty)}(s) + |\bar S|k \Bigr| \\
    &= \min_{k \in I}\lim_{i \to \infty}\Bigl|\pi^{(i)}(r) -
      \pi^{(i)}(s) + |\bar S_i|k \Bigr| \\
    &= \lim_{i \to \infty}\min_{k \in I}\Bigl|\pi^{(i)}(r) -
      \pi^{(i)}(s) + |\bar S_i|k \Bigr| \\
    &= \lim_{i \to \infty}\min_{k \in \Z}\Bigl|\pi^{(i)}(r) -
      \pi^{(i)}(s) + |\bar S_i|k \Bigr| \\
    &= \lim_{i \to \infty}d_{S_i}\Bigl(\pi^{(i)}(r) + |S_i|\Z,
      \pi^{(i)}(s) + |S_i|\Z\Bigr) \\
    &= \lim_{i \to \infty}d_{S_i}\Bigl(\pi^{(i)}\bigl(r+|S|\Z\bigr),
      \pi^{(i)}\bigl(s+|S|\Z\bigr)\Bigr)
  \end{align*}
  where
  $I = \biggl\{k \in \Z \sth |k| \le \Bigl\lceil\frac{|r-s|}{|\bar
    S|}\Bigr\rceil\biggr\}$.  So the pseudometric on $S$ pulled back
  from $\pi^{(\infty)}$ is $\delta$. Then, since $\pi^{(\infty)}$ is
  surjective, this implies that $\bar S$ is $S_{\infty}$, the induced
  metric quotient of $(S,\delta)$ and $\pi^{\infty}$ is the quotient
  map.
\end{proof}

\begin{rmk}
  If a completely disjoint tightening sequence of a Riemannian circle
  is eventually constant then, for large enough $i$, the limit
  Riemannian circle $S_{\infty}$ is isometric to the $i$th Riemannian
  circle of the sequence $S_i$.
\end{rmk}

\subsection{Tightening sequence for an \texorpdfstring{$R$}{R}-circle}

Let $R \ge 0$ and let $X$ be an $R$-rough geodesic metric space. Let
$\alpha \colon S \to X$ be an $R$-circle.  A \defterm{tightening
  sequence} for $\alpha$ is a sequence of intervals and Riemannian
circles $(P_i)_i$, a sequence of Riemannian circles $(S_i)_i$ and
sequences of maps as in the commutative diagram
\[ \begin{tikzcd}
    S \ar[r,equal] \ar[drrr,bend right=15,"\alpha"'] & S_0 \ar[drr,bend right=15,"\alpha_0"] & P_0 \ar[l,hook'] \ar[r] & S_1 \ar[d,"\alpha_1"'] & P_1 \ar[l,hook'] \ar[r] & S_2 \ar[dll,bend left=15,"\alpha_2"'] & P_2 \ar[l,hook'] \ar[r] & \ar[dllll,bend left=10,"\alpha_3"'] \cdots \\
    & & & X & & & & \cdots
  \end{tikzcd} \] such that each $\alpha_i$ is an $R$-circle and the
sequence of maps
\[ \begin{tikzcd} S = S_0 & P_0 \ar[l,hook'] \ar[r] & S_1 & P_1
    \ar[l,hook'] \ar[r] & S_2 & P_2 \ar[l,hook'] \ar[r] & \cdots
  \end{tikzcd} \] is a tightening sequence for $S$.  Then, by the
discussion of \Ssecref{tightening_riemannian_circle}, we have a
diagram \[ \begin{tikzcd}
    & Q_0 \ar[r,two heads] \ar[dd,hook] & \bar Q_0 \ar[ddr,bend left=15,hook] & Q_1 \ar[r,two heads] \ar[dd,hook] & \bar Q_1 \ar[ddr,bend left=15,hook] & Q_2 \ar[r,two heads] \ar[dd,hook] & \bar Q_2 \ar[ddr,bend left=15,hook] & \cdots \\
    \\
    S \ar[r,equal] \ar[drrr,bend right=15,"\alpha"'] & S_0 \ar[rr,bend left=45,two heads,"\pi_0"] \ar[drr,bend right=15,"\alpha_0"] & P_0 \ar[l,hook'] \ar[r] & S_1 \ar[d,"\alpha_1"'] \ar[rr,bend left=45,two heads,"\pi_1"] & P_1 \ar[l,hook'] \ar[r] & S_2 \ar[dll,bend left=15,"\alpha_2"'] \ar[rr,bend left=45,two heads,"\pi_2"] & P_2 \ar[l,hook'] \ar[r] & \ar[dllll,bend left=10,"\alpha_3"'] \cdots \\
    & & & X & & & & \cdots
  \end{tikzcd} \] where the bounded planar regions are commuting
triangles and squares.

\begin{lem}
  \lemlabel{lim_r_circle} Consider a tightening sequence for an
  $R$-circle $\alpha \colon S \to X$ in an $R$-rough geodesic metric
  space $X$, with the same notation as above.  If the tightening
  sequence is completely disjoint and the sum
  $\sum_{i=0}^{\infty} |Q_i|$ of the tightened segment lengths is
  strictly less than $|S|$ then
  \begin{align*}
    \alpha_{\infty} \colon S_{\infty} &\to X \\
    x &\mapsto \lim_{i \to \infty} \alpha_i \comp \pi^{(i)} \comp \sigma(x)
  \end{align*}
  defines an $R$-circle, called a \defterm{limit $R$-circle} of the
  tightening sequence, where $S_{\infty}$ is the limit Riemannian
  circle of the tightening sequence and
  $\sigma \colon S_{\infty} \to S$ is a section of the quotient map
  $S \to S_{\infty}$.
\end{lem}
\begin{proof}
  Since the tightening sequence is completely disjoint, we may think
  of the $Q_i$ as a collection of disjoint segments in $S$.  For
  $x \in S$ either $x \notin \bigcup_{i=0}^{\infty} Q_i$ and
  $\bigl(\alpha_i \comp \pi^{(i)}(x)\bigr)_{i=0}^{\infty}$ is a
  constant sequence or $x \in Q_j$ for some $j$ and the tail sequence
  $\bigl(\alpha_i \comp \pi^{(i)}(x)\bigr)_{i=j+1}^{\infty}$ is
  constant.  In either case, the limit
  $\lim_{i \to \infty} \alpha_i \comp \pi^{(i)}(x)$ exists so we have
  a function $\alpha_{\infty}' \colon S \to X$ given by
  $\alpha_{\infty}'(x) = \lim_{i \to \infty} \alpha_i \comp
  \pi^{(i)}(x)$.

  Let $S_{\infty}$ be the limit Riemannian circle given by
  \Lemref{lim_riemannian_circle}.  So $S_{\infty}$ is the induced
  metric quotient of $(S,\delta)$ where $\delta$ is the pseudometric
  give by
  $\delta(x,y) = \lim_{i \to \infty} d_{S_i}\bigl(\pi^{(i)}(x),
  \pi^{(i)}(y) \bigr)$.  Since each $\alpha_i$ is an $R$-circle, for
  $x,y \in S$,
  \[ d_X\bigl(\alpha_i\comp \pi^{(i)}(x), \alpha_i\comp
    \pi^{(i)}(y)\bigr) \le
    d_{S_i}\bigl(\pi^{(i)}(x),\pi^{(i)}(y)\bigr) + R \] for all $i$,
  and thus
  \[ d_X\bigl(\alpha_{\infty}'(x), \alpha_{\infty}'(y)\bigr) \le
    \delta(x,y) + R \] by taking limits as $i \to \infty$.

  Then, for $x,y \in S_{\infty}$
  \begin{align*}
    d_X\bigl(\alpha_{\infty}(x), \alpha_{\infty}(y)\bigr)
    &= d_X\Bigl(\alpha_{\infty}'\bigl(\sigma(x)\bigr),
      \alpha_{\infty}'\bigl(\sigma(y)\bigr)\Bigr) \\
    &\le \delta\bigl(\sigma(x), \sigma(y)\bigr) + R \\
    &= d_{S_{\infty}}(x, y) + R
  \end{align*}
  which completes the proof.
\end{proof}

\begin{rmk}
  If a completely disjoint tightening sequence for an $R$-circle
  $\alpha \colon S \to X$ is eventually constant then, for large
  enough $i$, the limit $R$-circle
  $\alpha_{\infty} \colon S_{\infty} \to X$ is isometric over $X$ to
  the $i$th $R$-circle of the sequence $\alpha_i$.  This means that
  there is an isometry $S_{\infty} \to S$ such that diagram
  \[ \begin{tikzcd}
      S_{\infty} \ar[r] \ar[dr,"\alpha_{\infty}"'] & S_i \ar[d,"\alpha_i"] \\
      & X
    \end{tikzcd} \]
  commutes.
\end{rmk}

We are ready now to state the Circle Tightening Lemma.

\begin{lem}[Circle Tightening Lemma]
  \lemlabel{circle_tightening} Let $N > 1$, let $L > 1$, let $K > 1$
  be small enough (depending on $N$ and $L$), let $R \ge 0$, let
  $M > 0$ be large enough (depending on $N$, $L$, $K$ and $R$) and let
  $C \ge 4R$.

  Let $\alpha \colon S \to X$ be an $R$-circle in an $R$-rough
  geodesic metric space.  If $\alpha$ is $\frac{1}{K}$-almost
  isometric and $|S| > M$ then $\alpha$ has a completely disjoint
  tightening sequence such that the total length
  $\sum_{i=0}^{\infty} |Q_i|$ of the tightened segments is at most
  $\frac{|S|}{N}$ and the limit $R$-circle $\alpha_{\infty}$ is an
  $(L,C)$-quasi-isometric embedding.  If, additionally, we have
  $C > 0$ then such a tightening sequence exists that is eventually
  constant.
\end{lem}

\Lemref{circle_tightening} is a consequence of claims
\claimref{greedy_tightening_disjoint}, \claimref{greedy_bounds},
\claimref{greedy_tightening_qi} and
\claimref{positive_c_eventually_constant} and the strict inequalities
of \Claimref{rational_functions} below but to understand these claims
we need to first define greedy tightening sequences and prove some
properties about them.

\subsection{Greedy tightening sequences}
\sseclabel{greedy_tightening}

In order to prove the \Lemref{circle_tightening} we will need to
describe a tightening sequence that is constructed inductively by
greedily choosing segments to tighten.  Let $\alpha \colon S \to X$ be
an $R$-circle in an $R$-rough geodesic metric space $X$, with
$R \ge 0$.  Let $C \ge 4R$ and let $L > 1$.  We will inductively
define a tightening sequence for $\alpha$ with the same notation as in
the previous sections.  Suppose we have $\alpha_i \colon S_i \to X$.
If $\alpha_i$ is an $(L,C)$-quasi-isometric embedding then we extend
the sequence as follows.
\begin{enumerate}
\item We set $\pc{i} = P_i = S_i = S_{i+1}$ and
  $Q_i = \bar Q_i = \emptyset$.
\item We let $P_i \hookrightarrow S_i$ and $P_i \to S_{i+1}$ be
  identity maps.
\item We let $\alpha_{i+1} = \alpha_i$.
\end{enumerate}
Otherwise, the set
\[ J_i = \Bigl\{(p,q) \in S_i \times S_i \sth
  d_X\bigl(\alpha_i(p),\alpha_i(q)\bigr) < \frac{1}{L} d_{S_i}(p,q) -
  C \Bigr\} \] is nonempty and $d_{S_i}(p,q) > LC \ge 0$ for any
$(p,q) \in J_i$ and so
$s_i = \sup \bigl\{d_{S_i}(p,q) \sth (p,q) \in J_i \bigr\} > 0$.  By
compactness of $S$, there is a sequence $(p_i^{(n)},q_i^{(n)})_n$ in
$J_i$ that converges to some $(p_i,q_i) \in S \times S$ with
$d_{S_i}(p_i,q_i) = s_i$ as $n \to \infty$.  Then
\begin{align*}
  &d_X\bigl(\alpha_i(p_i), \alpha(q_i)\bigr) \\
  &\le d_X\bigl(\alpha_i(p_i), \alpha(p_i^{(n)})\bigr)
    + d_X\bigl(\alpha_i(p_i^{(n)}), \alpha(q_i^{(n)})\bigr)
    + d_X\bigl(\alpha_i(q_i^{(n)}), \alpha(q_i)\bigr) \\
  &< d_{S_i}(p_i, p_i^{(n)}) + R
    + \frac{1}{L}d_{S_i}(p_i^{(n)}, q_i^{(n)}) - C
    + d_{S_i}(q_i^{(n)}, q_i) + R \\
  &\to \frac{1}{L}d_{S_i}(p_i, q_i) - C + 2R
\end{align*}
as $n \to \infty$.  So
\begin{equation}
  \eqnlabel{alph_qi_alph_pi_distance}
  d_X\bigl(\alpha_i(p_i), \alpha(q_i)\bigr)
  \le \frac{1}{L}d_{S_i}(p_i, q_i) - C + 2R
\end{equation}
and
\begin{equation}
  \eqnlabel{qi_size}
  \text{$d_{S_i}(p_i,q_i) > 0$ and $d_{S_i}(p_i,q_i) \ge L(C-2R)$}
\end{equation}
hold.  Let $Q_i$ be a geodesic segment of $S_i$ between $p_i$ and
$q_i$.  In the case where $p_i$ and $q_i$ are antipodal in $S_i$,
there are two geodesic segments between $p_i$ and $q_i$; in this case
we let $Q_i$ be the geodesic segment whose intersection with
$\pc{0,i}$ has greatest total length.  Let $\pc{i}$ be the complement
of $Q_i$ and let $P_i$ be the closure of $\pc{i}$.  Let
$\gamma'_i \colon \bar Q'_i \to X$ be an $R$-rough geodesic from
$\alpha_i(p_i)$ to $\alpha_i(q_i)$.  For $x \in X$, let
$c_x\colon [0,R] \to X$ denote the constant path of length $R$ at $x$.
Let $\gamma_i \colon \bar Q_i \to X$ be the concatenation
$c_{\alpha_i(p_i)}\gamma'_ic_{\alpha_i(q_i)}$.  We have
\[ 0 \le |\bar Q_i| - 2R = d_X\bigl(\alpha_i(p_i),\alpha_i(q_i)\bigr)
  \le \frac{1}{L} |Q_i| - C + 2R \] and so, since $C \ge 4R$,
\begin{equation}
  \eqnlabel{qi_scale}
  |\bar Q_i| \le \frac{1}{L} |Q_i|
\end{equation}
holds.  We obtain $\alpha_{i+1} \colon S_{i+1} \to X$ from
$\alpha_i|_{P_i}$ and $\gamma_i$ by identifying the corresponding
endpoints of $\bar Q_i$ with $p_i$ and $q_i$ in $P_i$.  Then, by
consideration of \Rmkref{rpath_concatenation}, the map $\alpha_{i+1}$
is an $R$-circle.
\begin{rmk}
  \rmklabel{endpoint_distances}
  The inequalities
  \[ d_X\bigl(\alpha_{i+1}(p_i), \alpha_{i+1}(x)\bigr) \le
    d_{S_{i+1}}(p_i,x) \] and
  \[ d_X\bigl(\alpha_{i+1}(q_i), \alpha_{i+1}(x)\bigr) \le
    d_{S_{i+1}}(q_i,x) \] hold for any $x \in \bar Q_i$.
\end{rmk}

This completes the description of our inductive construction.  Any
tightening sequence for $\alpha$ obtained in this way is called an
\defterm{$(L,C)$-greedy tightening sequence} for $\alpha$.  The
importance of this construction for us is evident from the following
claim.

\begin{claim}
  \claimlabel{greedy_tightening_qi} Let $L > 1$, let $K > 1$, let
  $C \ge 4R$ and consider an $(L,C)$-greedy tightening sequence for a
  $\frac{1}{K}$-almost isometric $R$-circle $\alpha \colon S \to X$ in
  an $R$-rough geodesic metric space $X$.  If the tightening sequence
  is completely disjoint and the sum $\sum_{i=0}^{\infty} |Q_i|$ of
  the tightened segment lengths is strictly less than $|S|$ then any
  limiting $R$-circle $\alpha_{\infty} \colon S_{\infty} \to X$ is an
  $(L,C)$-quasi-isometric embedding.
\end{claim}
\begin{proof}
  Let $S_{\infty}$ be the limit Riemannian circle given by
  \Lemref{lim_riemannian_circle}.  So $S_{\infty}$ is the induced
  metric quotient of $(S,\delta)$ where $\delta$ is the pseudometric
  given by
  $\delta(x,y) = \lim_{i \to \infty} d_{S_i}\bigl(\pi^{(i)}(x),
  \pi^{(i)}(y) \bigr)$.  Let $\alpha_{\infty} \colon S_{\infty} \to X$
  be a limit $R$-circle as in \Lemref{lim_r_circle}.  So
  $\alpha_{\infty}$ is an $R$-circle defined by
  $\alpha_{\infty}(x) = \lim_{i \to \infty} \alpha_i \comp \pi^{(i)}
  \comp \sigma(x)$, where $\sigma \colon S_{\infty} \to S$ is a
  section of the quotient map $S \to S_{\infty}$.

  If $\alpha_{\infty}$ is not an $(L,C)$-quasi-isometric embedding
  then, since $R \le C$,
  \[ d_X\bigl(\alpha_{\infty}(x), \alpha_{\infty}(y)\bigr) <
    \frac{1}{L} d_{S_{\infty}}(x,y) - C \] for some
  $x,y \in S_{\infty}$, which then must be distinct.  But
  \[ d_{S_{\infty}}(x,y) = \lim_{i \to \infty}
    d_{S_i}\Bigl(\pi^{(i)}\bigl(\sigma(x)\bigr),
    \pi^{(i)}\bigl(\sigma(y)\bigr) \Bigr)
  \]
  and
  \[ d_X\bigl(\alpha_{\infty}(x), \alpha_{\infty}(y)\bigr) = \lim_{i
      \to \infty} d_X\Bigl(\alpha_i \comp \pi^{(i)}\bigl(
    \sigma(x)\bigr), \alpha_i \comp \pi^{(i)}\bigl(
    \sigma(y)\bigr)\Bigr) \] where, by complete disjointness,
  $\Bigl(\pi^{(i)}\bigl( \sigma(x)\bigr)\Bigr)_i$ and
  $\Bigl(\pi^{(i)}\bigl( \sigma(y)\bigr)\Bigr)_i$ are eventually
  constant.  So, for all large enough $j$,
  \begin{align*}
    &d_X\Bigl(\alpha_j \comp \pi^{(j)}\bigl(
    \sigma(x)\bigr), \alpha_j \comp \pi^{(j)}\bigl(
    \sigma(y)\bigr)\Bigr) \\
    &= \lim_{i
      \to \infty} d_X\Bigl(\alpha_i \comp \pi^{(i)}\bigl(
      \sigma(x)\bigr), \alpha_i \comp \pi^{(i)}\bigl(
      \sigma(y)\bigr)\Bigr)
  \end{align*}
  but also, for all large enough $j$,
  \[ \lim_{i \to \infty} d_X\Bigl(\alpha_i \comp \pi^{(i)}\bigl(
    \sigma(x)\bigr), \alpha_i \comp \pi^{(i)}\bigl(
    \sigma(y)\bigr)\Bigr) < \frac{1}{L}
    d_{S_j}\Bigl(\pi^{(j)}\bigl(\sigma(x)\bigr),
    \pi^{(j)}\bigl(\sigma(y)\bigr) \Bigr) - C \] so, for all large
  enough $j$,
  \[ d_X\Bigl(\alpha_j\bigl(\pi^{(j)} \comp \sigma(x)\bigr),
    \alpha_j\bigl(\pi^{(j)} \comp \sigma(y)\bigr)\Bigr) < \frac{1}{L}
    d_{S_j}\Bigl(\pi^{(j)}\comp \sigma(x), \pi^{(j)}\comp \sigma(y)
    \Bigr) - C \] which implies that
  $\bigl(\pi^{(j)}\comp \sigma(x), \pi^{(j)}\comp \sigma(y)\bigr) \in
  J_j$, for all large enough $j$.  But then, for all large enough $j$,
  \[ d_{S_j}\Bigl(\pi^{(j)}\comp \sigma(x), \pi^{(j)}\comp \sigma(y)
    \Bigr) \le s_j = |Q_j| \] with $\lim_{j \to \infty}|Q_j| = 0$ so
  \[ d_{S_{\infty}}(x,y) = \lim_{j \to \infty}
    d_{S_j}\Bigl(\pi^{(j)}\comp \sigma(x), \pi^{(j)}\comp \sigma(y)
    \Bigr) = 0 \] a contradiction.
\end{proof}

\subsection{Eventual constantness and greedy tightening sequences}

Consider a greedy tightening sequence with notation as in
\Ssecref{greedy_tightening}.  Note that if $S_i = S_{i+1}$ for some
$i$ then $S_i = S_{i+1} = S_{i+2} = \cdots$ so the tightening sequence
is eventually constant.  Moreover, for any $i$ for which
$S_i \neq S_{i+1}$, we have
\begin{align*}
  |S_i| - |S_{i+1}|
  &= |Q_i| - |\bar Q_i| \\
  &\ge |Q_i| - \frac{1}{L}|Q_i| \\
  &= \Bigl(1- \frac{1}{L} \Bigr)|Q_i| \\
  &\ge \Bigl(1- \frac{1}{L} \Bigr)L(C-2R) \\
  & = (L-1)(C-2R)
\end{align*}
by \peqnref{qi_size} and \peqnref{qi_scale}.  If $R > 0$ then, since
$C \ge 4R$, we have $C > 2R > 0$.  If $R = 0$ then $C > 2R$ is
equivalent to $C > 0$.  Thus, if $C > 0$ then
$|S_i| - |S_{i+1}| \ge (L-1)(C-2R) > 0$.  This implies the following
claim.
\begin{claim}
  \claimlabel{positive_c_eventually_constant} If $R \ge 0$, $L > 1$
  and $C > 0$ then any $(L,C)$-greedy tightening sequence for an
  $R$-circle in an $R$-rough geodesic space is eventually constant.
\end{claim}

\subsection{Disjointness and greedy tightening sequences}

Consider a greedy tightening sequence with notation as in
\Ssecref{greedy_tightening}.  Assume that $\alpha$ is
$\frac{1}{K}$-almost isometric and that the tightening sequence is
disjoint up to $j$.  Recall that, by the discussion in
\Ssecref{tightening_riemannian_circle}, we may think of the $Q_i$,
with $i < j$, as disjoint subspaces of $S$.

If $i < j$ and $Q_i \neq \emptyset$ then, by \Lemref{global_to_local}
and \peqnref{alph_qi_alph_pi_distance},
\[ |Q_i| - \frac{K-1}{K}\cdot\frac{|S|}{2} - 2R \le
  d_X\bigl(\alpha(p_i),\alpha(q_i)\bigr) \le \frac{1}{L} |Q_i| - C + 2R \]
but $C \ge 4R$ so
\[ |Q_i| - \frac{K-1}{K}\cdot\frac{|S|}{2} \le \frac{1}{L} |Q_i| \]
for all $i < j$.  Hence, we have established the following claim.
\begin{claim}
  \claimlabel{onebound} If an $(L,C)$-greedy tightening sequence for a
  $\frac{1}{K}$-almost isometric $R$-circle is disjoint up to $j$ then
  \[ |Q_i| \le \biggl(\frac{K -1}{K} \cdot
    \frac{L}{L-1}\biggr)\frac{|S|}{2} \] for any $i < j$, where $Q_i$
  is the $i$th replaced segment of the tightening sequence.
\end{claim}
By this claim, we can find a pair of points $p,q$ in the closure of
$S \setminus \bigl(\bigcup_{i=1}^{j-1}Q_i\bigr)$ at distance
$d_S(p,q) \ge \frac{|S|}{2} - \frac{K -1}{K} \cdot
\frac{L}{L-1}\cdot\frac{|S|}{4}$.  Let $A_1$ and $A_2$ be the two
segments of $S$ between $p$ and $q$.  If
$I_1 = \{i < j \sth Q_i \subseteq A_1 \}$ then, since $\alpha_j$ is an
$R$-circle,
\begin{align*}
  d_X\bigl(\alpha(p),\alpha(q)\bigr)
  &\le d_{S_j}(p,q) + R \\
  &\le |A_1| - \sum_{i \in I_1}|Q_i| + \sum_{i \in I_1}|\bar Q_i| + R \\
  &\le |A_1| - \sum_{i \in I_1}|Q_i| + \frac{1}{L}\sum_{i \in I_1}|Q_i| + R \\
  &= |A_1| - \frac{L - 1}{L}\sum_{i \in I_1}|Q_i| + R
\end{align*}
where the last inequality follows by \peqnref{qi_scale}.  The same
corresponding relations also hold for $A_2$ and
$I_2 = \{i < j \sth Q_i \subseteq A_2 \}$ and so, by
\Lemref{global_to_local},
\begin{align*}
  &|S| - \frac{L-1}{L}\sum_{i<j}|Q_i| + 2R \\
  &\ge 2 d_X\bigl(\alpha(p),\alpha(q)\bigr) \\
  &\ge 2d_S(p,q) - \frac{K-1}{K} \cdot |S| - 4R \\
  &\ge |S| - \Bigl(\frac{K -1}{K} \cdot
    \frac{L}{L-1}\Bigr)\frac{|S|}{2} - \frac{K-1}{K} \cdot |S| - 4R
\end{align*}
which establishes the following claim.
\begin{claim}
  \claimlabel{sumbound} If an $(L,C)$-greedy tightening sequence for a
  $\frac{1}{K}$-almost isometric $R$-circle is disjoint up to $j$ then
  \[ \sum_{i<j}|Q_i| \le \biggl(\frac{K-1}{K} \cdot
    \frac{L(3L-2)}{2(L-1)^2}\biggr)|S| + \frac{6LR}{L-1} \] where
  $Q_i$ is the $i$th replaced segment of the tightening sequence.
\end{claim}

\Claimref{sumbound} implies that
\[ \sum_{i<j}|Q_i| \le \biggl(\frac{K-1}{K} \cdot
  \frac{L(3L-2)}{2(L-1)^2} + \frac{6LR}{|S|(L-1)} \biggr)|S| \] so if
$\frac{K-1}{K} \cdot \frac{L(3L-2)}{2(L-1)^2} < \frac{1}{N}$ then if
$|S| > M$ for some $M$ depending only on $K$, $L$, $R$ and $N$ then
$\frac{K-1}{K} \cdot \frac{L(3L-2)}{2(L-1)^2} + \frac{6LR}{|S|(L-1)} <
\frac{1}{N}$ and so $\sum_{i<j}|Q_i| < \frac{|S|}{N}$.  Since
$\frac{K-1}{K} \cdot \frac{L(3L-2)}{2(L-1)^2} < \frac{1}{N}$ is
equivalent to $K < \frac{NL(3L-2)}{(3N-2)L^2 - (2N-4)L - 2}$, we have
established the following claim.

\begin{claim}
  \claimlabel{greedy_bounds} Let $N > 1$, let $L > 1$, let $R \ge 0$
  and let $K > 1$ satisfy
  $K < \frac{NL(3L-2)}{(3N-2)L^2 - (2N-4)L - 2}$.  Then there exists
  an $M > 0$ such that if an $(L,C)$-greedy tightening sequence for a
  $\frac{1}{K}$-almost isometric $R$-circle $\alpha \colon S \to X$ of
  length $|S| > M$ is disjoint up to $j$ then
  \[ \sum_{i<j}|Q_i| < \frac{|S|}{N} \] where the $Q_i$ are the
  tightened segments of the tightening sequence.
\end{claim}

The following claim about rational functions has a short and
elementary proof.  We will make use of it below.

\begin{claim}
  \claimlabel{rational_functions} The inequalities
  \[ 1 < \frac{L\bigl(9L^2 -3L -4\bigr)}{7L^3 + 3L^2 - 10L + 2} \] and
  \[ 1 < \frac{NL(3L-2)}{(3N-2)L^2 - (2N-4)L - 2} \] and
  \[ \frac{L\bigl(9L^2 -3L -4\bigr)}{7L^3 + 3L^2 - 10L + 2} \le
    \frac{L(5L-4)}{3L^2 - 2} \] and
  \[ \frac{L\bigl(9L^2 -3L -4\bigr)}{7L^3 + 3L^2 - 10L + 2} \le
    \frac{L(7L-6)}{5L^2-2L-2} \] hold for any $L > 1$ and $N > 1$.
\end{claim}

The next claim is essential in proving disjointness of greedy
tightening sequences.

\begin{claim}
  \claimlabel{greedy_tightening_almost_disjoint} Let $K > 1$, let
  $L > 1$ and let $R \ge 0$.  There exists an $M > 0$ such that if $X$
  is an $R$-rough geodesic metric space and $\alpha \colon S \to X$ is
  a $\frac{1}{K}$-almost isometric $R$-circle with $|S| > M$ and
  \[ K < \frac{L\bigl(9L^2 -3L -4\bigr)}{7L^3 + 3L^2 - 10L + 2} \] and
  $C \ge 4R$ then any $(L,C)$-greedy tightening sequence for $\alpha$
  that is disjoint up to $j$ satisfies the following statement.  With
  notation as above, if $(p,q) \in J_j$ and $Q$ is a geodesic segment
  from $p$ to $q$ in $S_j$ then $Q \subset \pc{0,j}$, where we view
  $\pc{0,j}$ as a subspace of $S_j$ via the embedding
  $\pc{0,j} \hookrightarrow S_j$.
\end{claim}
\begin{proof}
  First we will show that $Q$ is not contained in $\bar Q_i$ for any
  $i < j$.  Recall that $\bar Q_i$ is the concatenation
  $A \bar Q'_i B$ where $\alpha_j|_{\bar Q'_i} \colon \bar Q'_i \to X$
  is an $R$-rough geodesic and $\alpha_j$ is constant on $A$ and $B$,
  each of which is isometric to $[0,R]$.  By \peqnref{qi_size}, we
  have $|Q| \ge (C-2R)L \ge (4R - 2R)L \ge 2R$ and $|Q| > 0$ so we
  cannot have $Q \subseteq A$ or $Q \subseteq B$.  We also cannot
  have $\bar Q_i' \subseteq Q \subseteq \bar Q_i$ since then,
  \begin{align*}
    |\bar Q_i| - 2R
    &= d_X\bigl(\alpha_i(p_i), \alpha_i(q_i)\bigr) \\
    &= d_X\bigl(\alpha_j(p), \alpha_j(q)\bigr) \\
    &< \frac{1}{L}|Q| - C \\
    &\le \frac{1}{L}|\bar Q_i| - C \\
    &< |\bar Q_i| - C
  \end{align*}
  which contradicts $C \ge 4R$.  So, if $Q \subseteq \bar Q_i$ then
  some endpoint of $Q$ is contained in $\bar Q'_i$.  But this
  implies that $\alpha_j|_{Q} \colon Q \to X$ is a $2R$-rough
  geodesic and so, by \peqnref{alph_qi_alph_pi_distance} and
  \peqnref{qi_size},
  \begin{align*}
    |Q| - 2R
    &\le d_X\bigl(\alpha_j(p), \alpha_j(q)\bigr) \\
    &< \frac{1}{L}|Q| - C \\
    &< |Q| - C
  \end{align*}
  which, again, contradicts $C \ge 4R$.  Thus we see that $Q$ is not
  contained in $\bar Q_i$ for an $i < j$.  Hence $Q$ intersects
  $\pc{0,j}$ nontrivially in $S_j$.

\begin{figure}[ht]
  \newcommand{\circrad}{1.2cm}
  \newcommand{\seglen}{3cm}

  \newcommand{\arc}[2]{
    \draw[thick] (#1:\circrad) arc[start angle=#1, end angle=#2, radius=\circrad]}
  \newcommand{\replarc}[4]{
    \draw[train] (#1:\circrad) arc[start angle=#1, end angle=#2, radius=\circrad];
    \draw (#1:\circrad) -- (#2:\circrad) coordinate[midway](m);
    \node at (#1/2+#2/2:\circrad+2.5ex) {#3};
    \path (m) +(180+#1/2+#2/2:1em) coordinate (M);
    \node at (M) {#4}}
  \newcommand{\replarcb}[4]{
    \draw[train] (#1:\circrad) arc[start angle=#1, end angle=#2, radius=\circrad] coordinate[midway](m);
    \node at (#1/2+#2/2:\circrad+1.5ex) {#3};
    \path (m) +(180+#1/2+#2/2:1em) coordinate (M);
    \node at (M) {#4}}
  \newcommand{\replarcc}[4]{
    \draw[red,thick] (#1:\circrad) -- (#2:\circrad) coordinate[midway](m);
    \node at (#1/2+#2/2:\circrad+1.5ex) {#3};
    \path (m) +(180+#1/2+#2/2:1em) coordinate (M);
    \node at (M) {#4}}

  \newcommand{\lox}{-1.1cm}
  \newcommand{\ltx}{1.1cm}
  
  \centering
  
  \[ \begin{tikzcd}[cramped,sep=small]
      \begin{tikzpicture}
        \node at (0,0) {$S_j$};
        
        \node at (1em-\circrad,0) {$Q$};
        \path (95:\circrad) -- (145:\circrad) coordinate[midway] (p);
        \path (95:\circrad-1em) -- (145:\circrad-1em) coordinate[midway] (pl);
        \coordinate (q) at (235:\circrad);
        \coordinate (ql) at (235:\circrad-1.5ex);
        \draw[green,double distance=4pt] (p)
        -- (145:\circrad)
        arc[start angle=145, end angle=165, radius=\circrad]
        -- (195:\circrad)
        arc[start angle=195, end angle=205, radius=\circrad]
        -- (220:\circrad)
        arc[start angle=220, end angle=235, radius=\circrad];
        
        \replarcc{-65}{-30}{}{};
        \arc{-30}{-5};
        \replarcc{-5}{55}{}{};
        \arc{55}{95};
        \replarcc{95}{145}{$\bar A_p$}{};
        \arc{145}{165};
        \replarcc{165}{195}{}{};
        \arc{195}{205};
        \replarcc{205}{220}{}{};
        \arc{220}{245};
        \replarcc{245}{275}{}{};
        \arc{275}{295};
        
        \node[vertex] at (p) {};
        \node[vertex] at (q) {};
        \node at (pl) {$p$};
        \node at (ql) {$q$};
      \end{tikzpicture}
      &
      \begin{tikzpicture}
        \coordinate (x1) at (\lox,\seglen*0.5);
        \coordinate (p) at (\lox,\seglen*0.325);
        \coordinate (x2) at (\lox,\seglen*0.15);
        \coordinate (x3) at (\lox,0);
        \coordinate (x4) at (\lox,-\seglen*0.21);
        \coordinate (x5) at (\lox,-\seglen*0.28);
        \coordinate (x6) at (\lox,-\seglen*0.39);
        \coordinate (q) at (\lox,-\seglen*0.5);
        \draw[green,double distance=4pt] (p) -- (q);
        \draw[red,thick] (x1) -- (x2);
        \draw[thick] (x2) -- (x3);
        \draw[red,thick] (x3) -- (x4);
        \draw[thick] (x4) -- (x5);
        \draw[red,thick] (x5) -- (x6);
        \draw[thick] (x6) -- (q);
        \node[vertex] at (p) {};
        \node[vertex] at (q) {};
        
        \draw[decoration=brace,decorate] ($(x1)+(1.5ex,0)$)
        -- coordinate[midway] (mid) ($(x2)+(1.5ex,1pt)$);
        \node at ($(mid)+(2.5ex,0)$) {$\bar A_p$};
        \draw[decoration=brace,decorate] ($(x2)+(1.5ex,-1pt)$)
        -- coordinate[midway] (mid) (\lox+1.5ex,-\seglen*0.5);
        \node at ($(mid)+(2.5ex,0)$) {$Q^{-}$};
        \draw[decoration=brace,decorate] (\lox-1.5ex,-\seglen*0.5)
        -- coordinate[midway] (mid) (\lox-1.5ex,\seglen*0.5);
        \node at ($(mid)-(2.5ex,0)$) {$Q^{+}$};
        
        \coordinate (y1) at (\ltx,\seglen*0.5);
        \coordinate (r) at (\ltx,\seglen*0.325);
        \coordinate (y2) at (\ltx,\seglen*0.15);
        \coordinate (y3) at (\ltx,0);
        \coordinate (y4) at (\ltx,-\seglen*0.21);
        \coordinate (y5) at (\ltx,-\seglen*0.28);
        \coordinate (y6) at (\ltx,-\seglen*0.39);
        \coordinate (s) at (\ltx,-\seglen*0.5);
        \draw[train] (y1) -- (y2);
        \draw[thick] (y2) -- (y3);
        \draw[train] (y3) -- (y4);
        \draw[thick] (y4) -- (y5);
        \draw[train] (y5) -- (y6);
        \draw[thick] (y6) -- (s);
        
        \draw[decoration=brace,decorate] ($(y2)-(1.5ex,-1pt)$)
        -- coordinate[midway] (mid) ($(y1)-(1.5ex,0)$);
        \node at ($(mid)-(2.5ex,0)$) {$A_p$};
        \draw[decoration=brace,decorate] (\ltx-1.5ex,-\seglen*0.5)
        -- coordinate[midway] (mid) ($(y2)-(1.5ex,1pt)$);
        \node at ($(mid)-(2.5ex,0)$) {$Q_0^{-}$};
        \draw[decoration=brace,decorate] (\ltx+1.5ex,\seglen*0.5)
        -- coordinate[midway] (mid) (\ltx+1.5ex,-\seglen*0.5);
        \node at ($(mid)+(2.5ex,0)$) {$Q_0^{+}$};
      \end{tikzpicture}
      \ar[l]
      \ar[r]
      &
      \begin{tikzpicture}
        \node at (0,0) {$S$};
        
        \replarcb{-65}{-30}{}{};
        \arc{-30}{-5};
        \replarcb{-5}{55}{}{};
        \arc{55}{95};
        \replarcb{95}{145}{$A_p$}{};
        \arc{145}{165};
        \replarcb{165}{195}{}{};
        \arc{195}{205};
        \replarcb{205}{220}{}{};
        \arc{220}{245};
        \replarcb{245}{275}{}{};
        \arc{275}{295};
      \end{tikzpicture}
  \end{tikzcd} \]
  
\caption{In the proof of \Claimref{greedy_tightening_almost_disjoint},
  from a geodesic segment $Q$ of $S_j$ (indicated by a green outline)
  we obtain $Q^{-1}$ by removing the interiors of any $\bar Q_i$ that
  are partially contained in $Q$.  We obtain $Q^{+} \to S_j$ by
  extending the inclusion $Q \hookrightarrow S_j$ to include full
  copies of any $\bar Q_i$ that are partially contained in $Q$.  From
  $Q^{-} \hookrightarrow S_j$ and $Q^{+} \to S_j$ we obtain
  $Q_0^{-} \hookrightarrow S$ and $Q_0^{+} \to S$ in $S = S_0$ by
  replacing any $\bar Q_i \hookrightarrow S_j$ with
  $Q_i \hookrightarrow S$.  The $\bar Q_i$ with $i < j$ are draw in
  red in $S_j$.  The $Q_i$ with $i < j$ are drawn with perpendicular
  markings in $S$.}
\figlabel{greedy_tightening_qplus_qminus}
\end{figure}

See \Figref{greedy_tightening_qplus_qminus}.  We will define a segment
$\bar A_p \subset S_j$ containing $p$ and a corresponding segment
$A_p \subset S$.  (Note that $\bar A_p$ does not denote the closure of
$A_p$ here.) If $p$ is contained in the interior of $\bar Q_i$ for
some $i < j$ then let $\bar A_p = \bar Q_i$ and let $A_p = Q_i$.
Otherwise, let $A_p = \bar A_p = \{p\}$.  Define $A_q$ and $\bar A_q$
similarly for $q$.  A priori, it is possible that
$\bar A_p = \bar A_q$.  Let $Q^{-}$ be obtained from $Q$ by
subtracting the interiors of $\bar A_p$ and $\bar A_q$ and let
$Q^{+} \to S_j$ extend $Q \hookrightarrow S_j$ so as to include a full
copy of $\bar A_p$ and a full copy of $\bar A_q$.  Let
$Q_0^{-} \subset S$ be obtained from $Q^{-} \subset S_j$ by replacing
any $\bar Q_i \subset Q^{-}$ with $Q_i \subset S$, for $i < j$.  Let
$Q_0^{+} \to S$ be obtained from $Q^{+} \to S_j$ by replacing any
$\bar Q_i \hookrightarrow S_j$, where $i < j$, with
$Q_i \hookrightarrow S$.

  Let $p^{+}$ and $q^{+}$ be the images of the endpoints of $Q_0^{+}$ in
  $S$, with $p^{+}$ the endpoint corresponding to $p$ and $q^{+}$
  the endpoint corresponding to $q$.  Let $p^{-}$ and $q^{-}$ be the
  endpoints of $Q_0^{-}$ in $S$, with $p^{-}$ the endpoint corresponding
  to $p$ and $q^{-}$ the endpoint corresponding to $q$.  Then we
  have
  \begin{align*}
    &d_X\bigl(\alpha(p^{+}), \alpha(q^{+})\bigr) \\
    &\le d_X\bigl(\alpha(p^{+}), \alpha_j(p)\bigr)
      + d_X\bigl(\alpha_j(p), \alpha_j(q)\bigr)
      + d_X\bigl(\alpha_j(q), \alpha(q^{+})\bigr) \\
    &\le d_{\bar A_p}(p^{+}, p)
      + d_X\bigl(\alpha_j(p), \alpha_j(q)\bigr)
      + d_{\bar A_q}(q, q^{+}) \\
    &< d_{\bar A_p}(p^{+}, p)
      + \frac{1}{L}|Q| - C
      + d_{\bar A_q}(q, q^{+}) \\
    &= d_{\bar A_p}(p^{+}, p)
      + \frac{1}{L}\bigl(d_{\bar A_p}(p, p^{-}) + |Q^{-}|
      + d_{\bar A_q}(q^{-}, q) \bigr)
      + d_{\bar A_q}(q, q^{+}) - C \\
    &\le d_{\bar A_p}(p^{+}, p)
      + d_{\bar A_p}(p, p^{-}) + \frac{1}{L}|Q^{-}| + d_{\bar A_q}(q^{-}, q)
      + d_{\bar A_q}(q, q^{+}) - C \\
    &= |\bar A_p| + \frac{1}{L}|Q^{-}| + |\bar A_q| - C \\
    &\le |\bar A_p| + \frac{1}{L}|Q_0^{-}| + |\bar A_q| - C \\
    &\le \frac{1}{L}|A_p| + \frac{1}{L}|Q_0^{-}| + \frac{1}{L}|A_q| - C \\
    &= \frac{1}{L}|Q_0^{+}| - C
  \end{align*}
  where the second inequality follows from \Rmkref{endpoint_distances}
  and the last inequality follows from \peqnref{qi_scale}.  By
  assumption, $Q$ nontrivially intersects at least one $\bar Q_i$,
  with $i < j$.  Let $m$ be minimal such that $Q$ nontrivially
  intersects $\bar Q_m$.  Then, since $Q$ intersects $\pc{0,j}$
  nontrivially, the image of $Q_0^{+} \to S_m$ must strictly contain
  $Q_m$.

  So, if $Q_0^{+} \to S_m$ was the inclusion of a geodesic segment, we
  would have $(p^{+},q^{+}) \in J_m$, which would contradict
  $d_{S_m}(p_m,q_m) = s_m$.  Thus $|Q_0^{+}| > \frac{|S_m|}{2}$.  But
  then
  \begin{align*}
    |Q_0^{+}|
    &> \frac{|S_m|}{2} \\
    &\ge \frac{|S|}{2} - \frac{1}{2}\sum_{i < m}|Q_i| \\
    &\ge \frac{|S|}{2} - \biggl(\frac{K-1}{K} \cdot
      \frac{L(3L-2)}{2(L-1)^2}\biggr)\frac{|S|}{2} - \frac{3LR}{L-1} \\
    &= \biggl(1 - \frac{K-1}{K} \cdot
      \frac{L(3L-2)}{2(L-1)^2} - \frac{6LR}{|S|(L-1)}\biggr)\frac{|S|}{2} \\
  \end{align*}
  by \Claimref{sumbound} while
  \[ |Q_0| \le \biggl(\frac{K -1}{K} \cdot
    \frac{L}{L-1}\biggr)\frac{|S|}{2} \] by \Claimref{onebound}.  So
  $|Q_0^{+}| \le |Q_0|$ would imply
  \[ 1 - \frac{K-1}{K} \cdot \frac{L(3L-2)}{2(L-1)^2} -
    \frac{6LR}{|S|(L-1)} < \frac{K -1}{K} \cdot \frac{L}{L-1} \]
  which is equivalent to the following inequality.
  \begin{equation}
    \eqnlabel{if_qp_le_q0}
    1 < \frac{K-1}{K} \cdot \frac{L(5L-4)}{2(L-1)^2} +
    \frac{6LR}{|S|(L-1)}
  \end{equation}
  By hypothesis and \Claimref{rational_functions}, we have
  $K < \frac{L(5L-4)}{3L^2 - 2}$ which is equivalent to
  $1 > \frac{K-1}{K} \cdot \frac{L(5L-4)}{2(L-1)^2}$ so if $|S| > M'$
  for some $M'$ depending only on $K$, $L$ and $R$ then we would have
  $1 > \frac{K-1}{K} \cdot \frac{L(5L-4)}{2(L-1)^2} +
  \frac{6LR}{|S|(L-1)}$ and this would contradict
  \peqnref{if_qp_le_q0}.  Hence, assuming $|S|$ is greater than this
  $M'$, we have $|Q_0^{+}| > |Q_0|$.

  Then if $Q_0^{+} \to S$ were the inclusion of a geodesic segment
  then we would have $(p^{+},q^{+}) \in J_0$ and this would contradict
  $d_S(p_0,q_0) = s_0$.  Thus $|Q_0^{+}| > \frac{|S|}{2}$.  On the
  other hand
  \begin{align*}
    &|Q_0^{+}| \\
    &= |Q_0^{-}| + |A_p| + |A_q| \\
    &\le |Q^{-}|  + \sum_{i<j} \bigl(|Q_i|-|\bar Q_i|\bigr) +
      |A_p| + |A_q| \\
    &\le \frac{|S_j|}{2} + \sum_{i<j} \bigl(|Q_i|-|\bar Q_i|\bigr) +
      |A_p| + |A_q| \\
    &= \frac{|S|}{2} +
      \frac{1}{2}\sum_{i<j} \bigl(|Q_i|-|\bar Q_i|\bigr) +
      |A_p| + |A_q| \\
    &\le \frac{|S|}{2} + \frac{1}{2}\sum_{i<j} |Q_i| + |A_p| + |A_q| \\
    &\le \frac{|S|}{2} +
      \biggl(\frac{K-1}{K} \cdot
      \frac{L(3L-2)}{4(L-1)^2}\biggr)|S|
      + \frac{3LR}{L-1}
      + \biggl(\frac{K -1}{K} \cdot \frac{L}{L-1}\biggr)|S| \\
    &= \frac{|S|}{2} +
      \frac{K -1}{K} \cdot \frac{L}{L-1} \cdot \biggl(
      \frac{3L-2}{4(L-1)} + 1\biggr)|S|
      + \frac{3LR}{L-1} \\
    &= \frac{|S|}{2} +
      \frac{K -1}{K} \cdot \frac{L}{L-1} \cdot
      \frac{7L- 6}{4(L-1)} \cdot |S|
      + \frac{3LR}{L-1} \\
    &= \biggl(\frac{1}{2} + \frac{K-1}{K} \cdot \frac{L(7L- 6)}{4(L-1)^2} 
       + \frac{3LR}{|S|(L-1)}\biggr) |S|
  \end{align*}
  where the last inequality follows from \Claimref{onebound} and
  \Claimref{sumbound}.  By hypothesis and
  \Claimref{rational_functions}, we have
  $K < \frac{L(7L-6)}{5L^2-2L-2}$, which is equivalent to
  $\frac{1}{2} + \frac{K-1}{K} \cdot \frac{L(7L- 6)}{4(L-1)^2} < 1$.
  Thus, if $|S| > M''$ for some $M''$ depending only on $K$, $L$ and
  $R$ then
  $\frac{1}{2} + \frac{K-1}{K} \cdot \frac{L(7L- 6)}{4(L-1)^2} +
  \frac{3LR}{|S|(L-1)} < 1$ and so $|Q_0^{+}| < |S|$ so that $Q_0^{+}$
  embeds in $S$.  In this case, the endpoints $p^{+},q^{+}$ of
  $Q_0^{+}$ in $S$ are at distance
  \[ d_S(p^{+},q^{+}) \ge \biggl(\frac{1}{2} - \frac{K-1}{K} \cdot
    \frac{L(7L- 6)}{4(L-1)^2} - \frac{3LR}{|S|(L-1)}\biggr)
    |S| \]
  but we also have
  \begin{align*}
    &d_X\bigl(\alpha(p^{+}),\alpha(q^{+})\bigr) \\
    &< \frac{1}{L}|Q_0^{+}| - C \\
    &\le \frac{1}{L}
      \biggl(\frac{1}{2} + \frac{K-1}{K} \cdot \frac{L(7L- 6)}{4(L-1)^2} 
      + \frac{3LR}{|S|(L-1)}\biggr) |S| - C
  \end{align*}
  which, by \Lemref{global_to_local} and $C \ge 4R$, implies
  \begin{align*}
    &\frac{1}{L}
      \biggl(\frac{1}{2} + \frac{K-1}{K} \cdot \frac{L(7L- 6)}{4(L-1)^2} 
      + \frac{3LR}{|S|(L-1)}\biggr) |S| \\
    &> \biggl(\frac{1}{2} - \frac{K-1}{K} \cdot
      \frac{L(7L- 6)}{4(L-1)^2} - \frac{3LR}{|S|(L-1)}\biggr) |S|
      - \frac{K-1}{K} \cdot \frac{|S|}{2}
  \end{align*}
  which is equivalent to the following inequality.
  \begin{equation}
    \eqnlabel{qplus_and_almost_isometric}
    \frac{K-1}{K} \cdot
    \frac{9L^2 - 3L -4}{2(L-1)^2}
    + \frac{6R(L+1)}{|S|(L-1)} > \frac{L-1}{L}
  \end{equation}
  By hypothesis, we have
  $K < \frac{L\bigl(9L^2 -3L -4\bigr)}{7L^3 + 3L^2 - 10L + 2}$ which
  is equivalent to
  $\frac{K-1}{K} \cdot \frac{9L^2 - 3L -4}{2(L-1)^2} < \frac{L-1}{L}$
  so if $|S| > M'''$ for some $M'''$ depending only on $K$, $L$ and
  $R$ then we would have
  $\frac{K-1}{K} \cdot \frac{9L^2 - 3L -4}{2(L-1)^2} +
  \frac{6R(L+1)}{|S|(L-1)} < \frac{L-1}{L}$ which contradicts
  \peqnref{qplus_and_almost_isometric}.

  Therefore, if $|S| > M = \max\{M', M'', M'''\}$, which depends only
  on $K$, $L$ and $R$ then assuming the existence of a $j$ for which
  $Q \not\subset \pc{0,j}$ leads us to a contradiction.
\end{proof}

\begin{claim}
  \claimlabel{greedy_tightening_disjoint} Let $K > 1$, let $L > 1$ and
  let $R \ge 0$.  There exists an $M > 0$ such that if $X$ is an
  $R$-rough geodesic metric space and $\alpha \colon S \to X$ is a
  $\frac{1}{K}$-almost isometric $R$-circle with $|S| > M$ and
  \[ K < \frac{L\bigl(9L^2 - 3L - 4\bigr)}{7L^3 + 3L^2 - 10L + 2} \]
  and $C \ge 4R$ then any $(L,C)$-greedy tightening sequence for
  $\alpha$ is completely disjoint.
\end{claim}
\begin{proof}
  Consider an $(L,C)$-greedy tightening sequence for $\alpha$ with
  notation as above.  For the sake of finding a contradiction, suppose
  $j \ge 1$ is the least integer with $Q_j \not\subset \pc{0,j}$.  As
  above we view the $\bar Q_i$ with $i < j$ as disjoint segments of
  $S_j$ with $S_j \setminus \bigcup_{i=0}^{j-1}\bar Q_m = \pc{0,j}$.

  Since $Q_j \not\subset \pc{0,j}$, we have
  $\bar Q_m \cap Q_j \neq \emptyset$, for some $m < j$.  Recall that
  $(p_j,q_j)$ is the limit of a sequence $(p_j^{(n)},q_j^{(n)})_n$
  in $J_j$.  For each $n$, let $Q_j^{(n)}$ be a geodesic segment
  between $p_j^{(n)}$ and $q_j^{(n)}$ in $S_j$.  By
  \Claimref{greedy_tightening_almost_disjoint}, we have
  $Q_j^{(n)} \subset \pc{0,j}$, so we have
  $\bar Q_m \cap Q_j \subset \{p_m,q_m\} \cap \{p_j,q_j\}$. Without
  loss of generality, we may assume $q_m = p_j$.  See
  \Figref{greedy_disjoint}.

\begin{figure}[ht]
  \newcommand{\circrad}{2cm}

  \newcommand{\arc}[2]{
    \draw (#1:\circrad) arc[start angle=#1, end angle=#2, radius=\circrad]}
  \newcommand{\replarc}[4]{
    \draw[train] (#1:\circrad) arc[start angle=#1, end angle=#2, radius=\circrad];
    \draw[red] (#1:\circrad) -- (#2:\circrad)
    coordinate[midway](m);
    \node at (#1/2+#2/2:\circrad+1ex+1em) {#3};
    \path (m) +(180+#1/2+#2/2:1em) coordinate (M);
    \node at (M) {#4}}

  \centering
  \begin{tikzpicture}
    \node at (-55:\circrad-1em) {$q_j$};
    \replarc{-55}{-20}{$Q_j$}{};
    \node at (-20:\circrad-2em) {$p_j=q_m$};
    \replarc{-20}{35}{$Q_m$}{};
    \node at (35:\circrad-1em) {$p_m$};
    \node[vertex] at (-55:\circrad) {};
    \node[vertex] at (-20:\circrad) {};
    \node[vertex] at (35:\circrad) {};
    
    \arc{35}{55};
    \replarc{55}{85}{}{};
    \arc{85}{110};
    \replarc{110}{135}{}{};
    \arc{135}{175};
    \replarc{175}{205}{}{};
    \arc{205}{245};
    \replarc{245}{275}{}{};
    \arc{275}{305};
  \end{tikzpicture}
  
  \caption{A greedy tightening sequence that is disjoint up to $j$ but
    not up to $j+1$ with $Q_j$ intersecting the prior $Q_i$ only at
    endpoints, as in the proof of
    \Claimref{greedy_tightening_disjoint}.}
  \figlabel{greedy_disjoint}
\end{figure}

  Since $Q_j^{(n)} \subset \pc{0,j}$, we may think of the $Q_j^{(n)}$
  as segments of $S_m$, by the embedding
  $\pc{0,j} \hookrightarrow S_m$.  Each $Q_j^{(n)}$ is a geodesic
  segment in $S_m$ since the complementary segment of $Q_j^{(n)}$ in
  $S_m$ is even longer than the complementary segment of $Q_j^{(n)}$
  in $S_j$.  Thus $(p_j^{(n)},q_j^{(n)})_n$ is a sequence in $J_m$.
  For each $n$, let $Q_m^{(n)}$ be a segment between $p_m^{(n)}$ and
  $q_m^{(n)}$ such that $(Q_m^{(n)})_n$ converges to $Q_m$ in
  Hausdorff distance.  The circular orders on the triples
  $(p_m^{(n)},q_m^{(n)},q_j^{(n)})$ and
  $(p_m^{(n)},p_j^{(n)},q_j^{(n)})$ are eventually constant and equal.
  Let $(A^{(n)})_n$ be a sequence of segments in $S_m$ from
  $p_m^{(n)}$ to $q_j^{(n)}$ such that $A^{(n)}$ eventually contains
  $q_m^{(n)}$ or, equivalently, eventually contains $p_j^{(n)}$.

  Let $\epsilon > 0$ satisfy $\epsilon < \frac{|Q_j|}{3}$ and
  $\epsilon \le \frac{LR}{L+1}$.  Then, for $n$ large enough,
  \begin{align*}
    |A^{(n)}|
    &\ge |Q_m^{(n)}| + |Q_j^{(n)}| - d_{S_m}(q_m^{(n)}, p_j^{(n)}) \\
    &> |Q_m| - \epsilon + |Q_j| - \epsilon - \epsilon \\
    &= |Q_m| + |Q_j| - 3\epsilon \\
    &> |Q_m|
  \end{align*}
  and
  \begin{align*}
    &d_X\bigl(\alpha_m(p_m^{(n)}), \alpha_m(q_j^{(n)})\bigr) \\
    &\le d_X\bigl(\alpha_m(p_m^{(n)}), \alpha_m(q_m^{(n)})\bigr)
      + d_X\bigl(\alpha_m(q_m^{(n)}), \alpha_m(p_j^{(n)})\bigr) \\
      &\;\;\;\;\;\; + d_X\bigl(\alpha_m(p_j^{(n)}), \alpha_m(q_j^{(n)})\bigr) \\
    &< \frac{1}{L}d_{S_m}(p_m^{(n)}, q_m^{(n)}) - C
      + d_{S_m}(q_m^{(n)}, p_j^{(n)}) + R
      + \frac{1}{L}d_{S_m}(p_j^{(n)}, q_j^{(n)}) - C \\
    &< \frac{1}{L}|Q_m| + \epsilon - C
      + \epsilon + R
      + \frac{1}{L}|Q_j| + \epsilon - C \\
    &= \frac{1}{L}(|Q_m| + |Q_j|) - 2C + R + 3\epsilon \\
    &< \frac{1}{L}(|A^{(n)}| + 3\epsilon) - 2C + R + 3\epsilon \\
    &= \frac{1}{L}|A^{(n)}| - 2C + R + \frac{3(L+1)}{L}\cdot\epsilon \\
    &\le \frac{1}{L}|A^{(n)}| - C
  \end{align*}
  since $C \ge 4R$.  So, if $A^{(n)}$ is a geodesic segment for
  arbitrarily large $n$ then, for some $n$, we would have
  $(p_m^{(n)},q_j^{(n)}) \in J_m$ and
  $d_{S_m}(p_m^{(n)},q_j^{(n)}) > |Q_m| = s_m$, which is a
  contradiction.  Thus eventually
  $|Q_m^{(n)}| + |Q_j^{(n)}| = |A^{(n)}| > \frac{|S_m|}{2}$ and so
  $|Q_m| + |Q_j| \ge \frac{S_m}{2}$.

  Then, by \Claimref{onebound}, we have
  \[ \frac{S_m}{2} \le 2\biggl(\frac{K - 1}{K} \cdot
  \frac{L}{L-1}\biggr)\frac{|S|}{2} \] while
  \begin{align*}
    \frac{|S_m|}{2}
    &\ge \frac{|S|}{2} - \frac{1}{2}\sum_{i < m}|Q_i| \\
    &\ge \frac{|S|}{2} - \biggl(\frac{K-1}{K} \cdot
      \frac{L(3L-2)}{2(L-1)^2}\biggr)\frac{|S|}{2} - \frac{3LR}{L-1} \\
    &= \biggl(1 - \frac{K-1}{K} \cdot
      \frac{L(3L-2)}{2(L-1)^2} - \frac{6LR}{|S|(L-1)}\biggr)\frac{|S|}{2}
  \end{align*}
  by \Claimref{sumbound}.  Combining these we obtain
  \[ \frac{K-1}{K} \cdot \frac{L(7L-6)}{2(L-1)^2} +
    \frac{6LR}{|S|(L-1)} \ge 1 \] but, by hypothesis and
  \Claimref{rational_functions}, we have
  $K < \frac{L(7L-6)}{5L^2 - 2L - 2}$, which is equivalent to
  $\frac{K-1}{K} \cdot \frac{L(7L-6)}{2(L-1)^2} < 1$ so, if $|S|$ is
  large enough (depending only on $K$, $L$ and $R$) then we have a
  contradiction.
\end{proof}

\section{The Fine Milnor-Schwarz Lemma}
\seclabel{fine_ms}

The Fine Milnor-Schwarz Lemma is a refinement of the Milnor-Schwarz
Lemma that gives finer control on the multiplicative constant of the
quasi-isometry.  In this section we will state and prove this version
of the Milnor-Schwarz Lemma.  As a consequence we will prove that
every strongly shortcut group has a strongly shortcut Cayley graph.  A
corresponding statement should hold for any rough approximability
invariant of metric spaces.

Let $(X,d)$ be a rough geodesic metric space.  Let $G$ be a group
acting coboundedly on $X$ by isometries.  Let
\[ S_{x_0,t} = \bigl\{ g \in G \sth d(x_0,gx_0) \le t \bigr\} \] for
$x_0 \in X$ and $t \in \R_{\ge 0}$.

\begin{rmk}
  \rmklabel{metric_proper_finite} If the action of $G$ is metrically
  proper then the $S_{x_0,t}$ are finite.
\end{rmk}

Let $\Gamma_{x_0,t}$ be the graph with vertex set $G$ and with an edge
of length $t$ between $g$ and $g'$ whenever $g' = gs$ for some
$s \in S_{x_0,t}$.  So when $S_{x_0,t}$ generates $G$, the graph
$\Gamma_{x_0,t}$ is the Cayley graph of $G$ for the generating set
$S_{x_0,t}$ with edges scaled by $t$.  Let $d_{x_0,t}$ be the graph
metric on $\Gamma_{x_0,t}$ where we set $d_{x_0,t}(g,h) = \infty$ when
$g$ and $h$ are in different components of $\Gamma_{x_0,t}$.  Then,
when $S_{x_0,t}$ generates $G$, the metric $d_{x_0,t}$ is the word
metric on $G$ for the generating set $S_{x_0,t}$ scaled by $t$.  Let
$f_{x_0,t} \colon (G,d_{x_0,t}) \to (X,d)$ be defined by
$f_{x_0,t}(g) = gx_0$.  Let $K_{x_0,t}$ be the infimum of all
$K \ge 1$ for which there exists some $C_K > 0$ such that $f_{x_0,t}$
is a $(K,C_K)$-quasi-isometry.  Note that if $f_{x_0,t}$ is not a
quasi-isometry (e.g. if $S_{x_0,t}$ does not generate $G$) then
$K_{x_0,t} = \infty$.

\begin{lem}[Fine Milnor-Schwarz Lemma]
  \lemlabel{fine_ms} Let $(X,d)$ be a rough geodesic metric space.
  Let $G$ be a group acting coboundedly on $X$ by isometries.  Let
  $x_0 \in X$ and let $K_{x_0,t}$ be defined as above for
  $t \in \R_{\ge 0}$.  Then $K_{x_0,t} \to 1$ as $t \to \infty$.
\end{lem}
\begin{proof}
  Let $g \in G$.  We will prove that
  $d\bigl(f_{x_0,t}(1),f_{x_0,t}(g)\bigr) \le d_{x_0,t}(1,g)$.  If
  $d_{x_0,t}(1,g) = \infty$ there is nothing to show and so
  $d_{x_0,t}(1,g) = Mt$ for some $M \in \N_{\ge 0}$ and there is a
  combinatorial path defined by
  \[ 1 = g_0, g_1, g_2, \ldots, g_M = g \] in $\Gamma_{x_0,t}$.  The
  $g_{i-1}^{-1}g_i$ are contained in $S_{x_0,t}$ and so, by the
  triangle inequality, we have the following.
  \begin{align*}
    d\bigl(f_{x_0,t}(1),f_{x_0,t}(g)\bigr)
    &= d(x_0,gx_0) \\
    &\le d(x_0,g_1x_0) + d(g_1x_0,g_2x_0) + \ldots + d(g_{M-1}x_0,gx_0) \\
    &= d(x_0,g_1x_0) + d(x_0,g_1^{-1}g_2x_0) + \ldots + d(x_0,g_{M-1}^{-1}gx_0) \\
    &\le Mt \\
    &= d_{x_0,t}(1,g)
  \end{align*}

  We now establish a lower bound on
  $D = d\bigl(f_{x_0,t}(1),f_{x_0,t}(g)\bigr)$.  Since $G$ acts
  coboundedly on $X$, the orbit $Gx_0$ is a quasi-onto subspace of
  $X$.  Hence $Gx_0$ is roughly isometric to $X$ and so $Gx_0$ is also
  a rough geodesic metric space.  Let $R$ be the rough geodesicity
  constant of $Gx_0$.  Let $\alpha \colon [0,D] \to Gx_0$ be an
  $R$-rough geodesic from $f_{x_0,t}(1) = x_0$ to
  $f_{x_0,t}(g) = gx_0$.  Assume that $t > R$ and subdivide $[0,D]$
  into at most $\bigl\lceil\frac{D}{t-R}\bigr\rceil$ segments of
  length at most $t-R$.  Let $0 = a_0 < a_1 < a_2 < \cdots < a_M = D$
  be the endpoints of the segments and let $g_ix_0 = \alpha(a_i)$, for
  each $i$, with $g_0 = 1$ and $g_M = g$.  Then
  \begin{align*}
    d(x_0, g_i^{-1}g_{i+1}x_0)
    &= d(g_ix_0, g_{i+1}x_0) \\
    &= d\bigl(\alpha(a_i), \alpha(a_{i+1})\bigr) \\
    &\le |a_i - a_{i+1}| + R \\
    &\le t
  \end{align*}
  for each $i$.  Thus $g_i^{-1}g_{i+1} \in S_{x_0,t}$, for each $i$,
  and so
  \[ 1 = g_0, g_1, g_2, \ldots, g_{M-1}, g_M = g \] defines a
  combinatorial path in $\Gamma_{x_0,t}$.  Hence
  \begin{align*}
    d_{x_0,t}(1,g)
    &\le tM \\
    &\le t\Bigl\lceil\frac{D}{t-R}\Bigr\rceil \\
    &= t\biggl\lceil\frac{d\bigl(f_{x_0,t}(1),f_{x_0,t}(g)\bigr)}{t-R}\biggr\rceil \\
    &\le t\biggl(\frac{d\bigl(f_{x_0,t}(1),f_{x_0,t}(g)\bigr)}{t-R} + 1\biggr) \\
    &= \frac{t}{t-R}d\bigl(f_{x_0,t}(1),f_{x_0,t}(g)\bigr) + t
  \end{align*}
  and so we have
  $\frac{t-R}{t}d_{x_0,t}(1,g) - (t-R) \le
  d\bigl(f_{x_0,t}(1),f_{x_0,t}(g)\bigr)$.

  For $g,g' \in G$,
  \begin{align*}
    d\bigl(f_{x_0,t}(g),f_{x_0,t}(g')\bigr)
    &= d(gx_0, g'x_0) \\
    &= d(x_0,g^{-1}g'x_0) \\
    &= d\bigl(f_{x_0,t}(1),f_{x_0,t}(g^{-1}g')\bigr)
  \end{align*}
  and $d_{x_0,t}(g,g') = d_{x_0,t}(1,g^{-1}g')$ so we have
  \[ \frac{t-R}{t}d_{x_0,t}(g,g') - (t-R) \le
    d\bigl(f_{x_0,t}(g),f_{x_0,t}(g')\bigr) \le d_{x_0,t}(g,g') \]
  hence $1 \le K_{x_0,t} \le \frac{t-R}{t}$ when $t > R$.  This
  implies that $K_{x_0,t} \to 1$ as $t \to \infty$.
\end{proof}

\begin{cor}[Milnor-Schwarz Lemma]
  If the $G$-action on $X$ metrically proper then $G$ is finitely
  generated and, for any finite generating set $S$, the map
  \begin{align*}
    f_s \colon (G,d_S) &\to X \\
    g &\mapsto gx_0
  \end{align*}
  is a quasi-isometry, where $d_S$ is the word metric for the
  generating set $S$.
\end{cor}
\begin{proof}
  By \Lemref{fine_ms}, if $t$ is large enough then
  $K_{x_0,t} < \infty$.  Thus $S_{x_0,t}$ generates $G$.  By
  \Rmkref{metric_proper_finite}, the generating set $S_{x_0,t}$ is
  finite so $G$ is finitely generated.  Since $K_{x_0,t} < \infty$ and
  since scaling the metric $d_S$ by a factor of $t$ preserves the
  quasi-isometry type, the map $f_S$ is a quasi-isometry for
  $S = S_{x_0,t}$.  But the identity map on $G$ is a bilipschitz
  equivalence from $(G, d_{S'})$ to $(G, d_{S''})$ where $S'$ and
  $S''$ are any two finite generating sets and $d_{S'}$ and $d_{S''}$
  are the corresponding word metrics.  Thus $f_S$ is a quasi-isometry
  for any generating set $S$.
\end{proof}

\begin{cor}
  Let $G$ be a group.  If $G$ acts metrically properly and coboundedly
  on a strongly shortcut rough geodesic metric space $X$ then $G$ has
  a finite generating set $S$ for which the Cayley graph of $(G,S)$ is
  strongly shortcut.  In particular, the group $G$ is strongly
  shortcut.
\end{cor}
\begin{proof}
  By \Corref{rough_approx_inv}, there exists an $L_X > 1$ such that
  whenever $C > 0$ and $Y$ is a rough geodesic metric space and
  $f \colon Y \to X$ is an $(L_X,C)$-quasi-isometry up to scaling,
  then $Y$ is strongly shortcut.  But, by \Lemref{fine_ms}, there is a
  Cayley graph $\Gamma$ of $G$, a $C > 0$ and an
  $(L_X,C)$-quasi-isometry up to scaling $f \colon \Gamma \to X$.  So
  $\Gamma$ is strongly shortcut as a rough geodesic metric space.
  Then, by \Rmkref{ss_asgraph_asmetric}, the Cayley graph $\Gamma$ is
  strongly shortcut as a graph.
\end{proof}

\begin{cor}
  \corlabel{ss_group} Let $G$ be a group.  The following conditions
  are equivalent
  \begin{enumerate}
  \item $G$ is strongly shortcut.
  \item $G$ acts metrically properly and coboundedly on a strongly
    shortcut rough geodesic metric space.
  \item $G$ has a finite generating set $S$ for which the Cayley graph
    of $(G,S)$ is strongly shortcut.
  \end{enumerate}
\end{cor}

\section{Asymptotically \texorpdfstring{$\CAT(0)$}{CAT(0)} spaces}
\seclabel{as_cat0}

In this section we will apply the characterizations of
\Secref{characterizations} to prove that asymptotically $\CAT(0)$
rough geodesic metric spaces are strongly shortcut.  By the results of
\Secref{fine_ms}, this will imply that asymptotically $\CAT(0)$ groups
are strongly shortcut.

Asymptotically $\CAT(0)$ spaces and groups were first introduced and
studied by Kar \cite{Kar:2011}.  A metric space $X$ is
\defterm{asymptotically $\CAT(0)$} if every asymptotic cone of $X$ is
$\CAT(0)$.  A group is \defterm{asymptotically $\CAT(0)$} if it acts
properly and cocompactly on an asymptotically $\CAT(0)$ proper
geodesic metric space.  (Note that the condition that the action be on
a proper metric space does not make this definition more restrictive
than the definition given by Kar since the definition of proper action
in Kar \cite{Kar:2011} seems to be that of Bridson and Haefilger
\cite[Section~I.8.2]{Bridson:1999} \cite[Page 77]{Kar:2011} and any
geodesic metric space admitting a cocompact action that is proper by
this more restricted definition is a proper metric space.)  For an
introduction to $\CAT(0)$ geodesic metric spaces, see Bridson and
Haefliger \cite{Bridson:1999}.

\begin{thm}
  \thmlabel{as_cat0_space} Asymptotically $\CAT(0)$ rough geodesic
  metric spaces are strongly shortcut.
\end{thm}
\begin{proof}
  By uniqueness of geodesics in $\CAT(0)$ geodesic metric spaces,
  there is no isometric copy of a Riemannian circle in the asymptotic
  cone of an asymptotically $\CAT(0)$ metric space $X$.  So, by
  \Thmref{not_sshortcut_equiv}, any asymptotically $\CAT(0)$ rough
  geodesic metric space is strongly shortcut.
\end{proof}

\begin{thm}
  \thmlabel{as_cat0_group} Asymptotically $\CAT(0)$ groups are
  strongly shortcut.
\end{thm}
\begin{proof}
  A proper and cocompact action on a proper metric space is metrically
  proper and cobounded.  So, by \Thmref{as_cat0_space}, any asymptotically
  $\CAT(0)$ group $G$ acts metrically properly and cobounded on a
  strongly shortcut geodesic metric space.  Thus, by
  \Corref{ss_group}, the group $G$ is strongly shortcut.
\end{proof}

\bibliographystyle{abbrv}
\bibliography{nima,fsss}
\end{document}
